\documentclass{amsart}

\usepackage[utf8]{inputenc}
\usepackage[charter]{mathdesign}
\usepackage{amsthm}
\usepackage{hyperref}

\hypersetup{
	pdfborder={0 0 0}
}

\newcommand{\N}{\mathbb{N}}

\newcommand{\Psf}{\mathsf{P}}
\newcommand{\Qsf}{\mathsf{Q}}

\newcommand{\da}{\!\downarrow}
\newcommand{\ua}{\!\uparrow}
\newcommand{\imp}{\rightarrow}

\newcommand{\biimp}{\leftrightarrow}

\newcommand{\Acal}{\mathcal{A}}

\newcommand{\Mcal}{\mathcal{M}}
\newcommand{\Pcal}{\mathcal{P}}

\newcommand{\Rcal}{\mathcal{R}}

\newcommand{\uh}{{\upharpoonright}}

\renewcommand{\setminus}{\smallsetminus}

\newcommand{\set}[1]{\left\{ #1 \right\}}
\newcommand{\card}[1]{\left| #1 \right|}

\newcommand{\tuple}[1]{\left\langle #1 \right\rangle}

\newcommand{\s}[1]{\sf{#1}}
\newcommand{\ran}[1]{#1\mbox{-}\s{RAN}}

\newcommand{\rca}{\s{RCA}_0}
\newcommand{\aca}{\s{ACA}_0}
\newcommand{\wkl}{\s{WKL}_0}

\newcommand{\dnr}{\s{DNR}}

\newcommand{\bst}{\s{B}\Sigma^0_2}
\newcommand{\ist}{\s{I}\Sigma^0_2}

\newcommand{\rwkl}{\s{RWKL}}
\newcommand{\rwwkl}{\s{RWWKL}}

\newcommand{\rt}{\s{RT}}
\newcommand{\srt}{\s{SRT}}

\newcommand{\rrt}{\s{RRT}}
\newcommand{\srrt}{\s{SRRT}}

\newcommand{\ads}{\s{ADS}}
\newcommand{\sads}{\s{SADS}}

\newcommand{\cac}{\s{CAC}}
\newcommand{\scac}{\s{SCAC}}

\newcommand{\coh}{\s{COH}}

\newcommand{\fip}{\s{FIP}}
\newcommand{\ts}{\s{TS}}
\newcommand{\sts}{\s{STS}}

\newcommand{\fs}{\s{FS}}
\newcommand{\sfs}{\s{SFS}}
\newcommand{\emo}{\s{EM}}
\newcommand{\semo}{\s{SEM}}

\newcommand{\opt}{\s{OPT}}
\newcommand{\amt}{\s{AMT}}

\newcommand{\sdnr}{\s{SDNR}}

\newcommand{\kl}{\s{KL}}

\newtheoremstyle{custom}
  {10pt}
  {10pt}
  {\normalfont}
  {}
  {\bfseries}
  {}
  { }
  {}

\theoremstyle{custom}

\newtheorem{theorem}{Theorem}[section]
\newtheorem{lemma}[theorem]{Lemma}
\newtheorem{definition}[theorem]{Definition}

\newtheorem{question}[theorem]{Question}

\newtheorem{corollary}[theorem]{Corollary}

\theoremstyle{remark}
\newtheorem*{remark}{Remark}
\newtheorem*{claim}{Claim}

\author{Ludovic Patey}

\address{Laboratoire PPS, Université Paris Diderot - Paris 7\\
Case 7014, 75205 Paris Cedex 13, France}

\email{ludovic.patey@computability.fr}

\title[Degrees bounding principles and universal instances]
{Degrees bounding principles\\ and universal instances in reverse mathematics}

\date{\today}

\begin{document}

\maketitle

\begin{abstract}
A Turing degree $\textbf{d}$ \emph{bounds} a principle $\mathsf{P}$ of reverse mathematics
if every computable instance of $\mathsf{P}$ has a $\textbf{d}$-computable solution.
$\mathsf{P}$ admits a \emph{universal instance} if there exists a computable instance
such that every solution bounds $\mathsf{P}$.
We prove that the stable version of the ascending descending sequence principle ($\sads$)
as well as the stable version of the thin set theorem for pairs ($\sts(2)$) do not admit
a bound of low${}_2$ degree. Therefore no principle between Ramsey's theorem for pairs ($\rt^2_2$)
and $\sads$ or $\sts(2)$ admit a universal instance.
We construct a low${}_2$ degree bounding the Erd\H{o}s Moser theorem $(\emo)$,
thereby showing that previous argument does not hold for $\emo$.
Finally, we prove that the only $\Delta^0_2$
degree bounding a stable version of the rainbow Ramsey theorem for pairs ($\srrt^2_2$)
is $\mathbf{0}'$. Hence no principle between the stable Ramsey theorem for pairs $(\srt^2_2)$
and $\srrt^2_2$ admit a universal instance. In particular the stable version of the Erd\H{o}s Moser theorem does
not admit one. It remains unknown whether $\emo$ admits a universal instance.
\end{abstract}

\section{Introduction}

Reverse mathematics is a program whose goal is to classify theorems
in function of their computational strength, within the framework
of subsystems of second order arithmetic. Proofs are done relatively
to a very weak system ($\rca$) meant to capture \emph{computational mathematics}.
$\rca$ is composed of basic Peano axioms, $\Delta^0_1$ comprehension and $\Sigma^0_1$ induction schemes.
See~\cite{hirschfeldt2013slicing} for a good introductory book.
Most of statements in reverse mathematics are of the form
$$
\forall X(\Phi(X) \imp \exists Y\Psi(X,Y))
$$
where $\Phi$ and $\Psi$ are arithmetical formulas.

A set $X$ such that $\Phi(X)$ holds is called an \emph{instance} of $\Psf$ 
and a set $Y$ such that $\Psi(X,Y)$ holds is a \emph{solution} to~$X$.
We can see relations between two instances $X_1, X_2$ of a statement $\Psf$
as a mass problem consisting of computing a solution to $X_1$ given any solution to~$X_2$.

\begin{definition}
Given a statement $\Psf$, a degree $\textbf{d}$ is \emph{$\Psf$-bounding} ($\textbf{d} \gg_{\Psf} \emptyset$)
if every computable instance $X$ of $\Psf$ has a $\textbf{d}$-computable solution.
A statement $\Psf$ admits a \emph{universal instance} if it has a computable
instance $X$ such that every solution to $X$ bounds $\Psf$.
\end{definition}

The notation $\textbf{d} \gg \emptyset$ historically means that the degree $\textbf{d}$
is PA and therefore is equivalent to $\textbf{d} \gg_{\wkl} \emptyset$
where $\wkl$ is the weak König's lemma principle, i.e. König's lemma restricted to subtrees
of $2^{<\omega}$.
It is well-known that $\wkl$ admits a universal instance 
-- e.g. take the $\Pi^0_1$ class of completions of Peano arithmetics --.
A few principles have been proven to admit universal instances -- $\wkl$~\cite{odifreddi1992classical}, 
König's lemma ($\kl$) \cite{hirschfeldt2013slicing}, the Ramsey-type weak
weak König's lemma ($\rwwkl$) \cite{bievenu2014rwkl}, 
the finite intersection property ($\fip$) \cite{downey2012fip}, 
the omitting partial type theorem ($\opt$) \cite{hirschfeldt2009atomic},
or even the rainbow Ramsey theorem for pairs
($\rrt^2_2$) \cite{millercom1} -- but most of principles do not admit one. An important notion
for proving such a result is computable reducibility.

\begin{definition}
A statement $\Psf$ is \emph{computably reducible } to a statement $\Qsf$ (written $\Psf \leq_c \Qsf$)
if for every instance $X$ of $\Psf$ there exists an instance $Y$ of $\Qsf$ computable from $X$
such that each solution to $Y$ computes relative to $X$ a solution to $X$.
\end{definition}

Mileti proved in~\cite{miletipartition} that the stable Ramsey theorem for pairs
($\srt^2_2$) admits no bound of low${}_2$ degree.
Therefore every statement $\Psf$ having an $\omega$-model
with only low${}_2$ sets, and such that $\srt^2_2 \leq_c \Psf$, admits no universal instance.
In particular none of Ramsey's theorem for pairs ($\rt^2_2$), $\srt^2_2$ and the Ramsey-type
weak König's lemma relative to $\emptyset'$ ($\rwkl[\emptyset']$) admit a universal instance.
Independently, Hirschfeldt \& Shore proved in~\cite{hirschfeldt2007combinatorial} 
that the stable ascending descending sequence principle ($\sads$) admits no bound of low degree.
Hence none of $\sads$ and the stable chain antichain principle ($\scac$) admit a universal instance.

We generalize both results by proving that $\sads$ admits not bound of low${}_2$ degree,
proving therefore that if a statement $\Psf$ has an $\omega$-model with only low${}_2$ sets and $\sads \leq_c \Psf$
then $\Psf$ admits no universal instance.
We also extend the result to statements to which the stable thin set theorem for pairs ($\sts(2)$) computably reduces.
Hence we deduce that none of the ascending descending sequence principle ($\ads$), 
the chain antichain principle ($\cac$), the thin set theorem for pairs ($\ts(2)$), 
the free set theorem for pairs ($\fs(2)$) and their
stable versions admit a universal instance.

We generalize the result to arbitrary tuples and prove that none of $\rt^n_2$,
$\fs(n)$, $\ts(n)$ and their stable versions admit a universal instance for $n \geq 2$.
The question remains open for the rainbow Ramsey theorem for $n$-tuples ($\rrt^n_2$) with $n \geq 3$.
We construct a low${}_2$ degree bounding the Erd\H{o}s Moser theorem $(\emo)$,
thereby showing that previous argument does not hold for~$\emo$.

Mileti proved in~\cite{miletipartition} that the only $\Delta^0_2$ degree bounding
$\srt^2_2$ is $\mathbf{0}'$. Using the fact that every $\Delta^0_2$ set has an infinite incomplete
$\Delta^0_2$ subset in either it or its complement \cite{hirschfeldt2008strength},
we obtain another proof that $\srt^2_2$ admits no universal instance. 
We extend this result by proving that the only $\Delta^0_2$ degree bounding
a stable version version of the rainbow Ramsey theorem for pairs ($\srrt^2_2$) 
is $\mathbf{0}'$. Hence none of the statements $\Psf$ satisfying $\srrt^2_2 \leq_c \Psf \leq_c \srt^2_2$
admit a universal instance. In particular we deduce that neither $\srrt^2_2$ nor 
the stable version of the Erd\H{o}s Moser theorem ($\semo$) admits a universal instance.

\subsection{Notations}

\emph{Formulas}.
The notation $(\forall^{\infty} s)\varphi(s)$ means that $\varphi(s)$ holds for all but finitely many~$s$,
i.e.\ is translated to $(\exists s_0)(\forall s \geq s_0)\varphi(s)$.
Given two sets $X$ and $Y$, we denote by $X \subseteq^{*} Y$ the statement $(\forall^{\infty} s \in X)[s \in Y]$.
Accordingly, $X =^{*} Y$ means that both $X \subseteq^{*} Y$ and $Y \subseteq^{*} X$ hold, 
i.e.\ $X$ and $Y$ differ by finitely many elements.

\emph{Turing functional and lowness}. 
We fix an effective enumeration of all Turing functionals $\Phi_0, \Phi_1, \dots$
We denote by $\Phi_{e,s}$ the partial approximation of the Turing functional $\Phi_e$ at stage $s$.
Given a set $X$, we denote by $X'$ the jump of $X$ and by $X^{(n)}$ the $n$th jump of $X$.
A set $X$ is \emph{low${}_n$ over $Y$} if $(X \oplus Y)^{(n)} \leq Y^{(n)}$.
A set is \emph{low${}_n$} if it is low${}_n$ over $\emptyset$.
A \emph{low${}_n$-ness} index of a set $X$ low${}_n$ over $Y$ is a Turing index $e$
such that $\Phi^{Y^{(n)}}_e = (X \oplus Y)^{(n)}$.

\emph{Mathias forcing}.
Given two sets $E$ and $F$, we denote by $E < F$ the formula
$(\forall x \in E)(\forall y \in F) x < y$.
A \emph{Mathias condition} is a pair $(F, X)$
where $F$ is a finite set, $X$ is an infinite set
and $F < X$.
A condition $(\tilde{F}, \tilde{X})$ \emph{extends } $(F, X)$ (written $(\tilde{F}, \tilde{X}) \leq (F, X)$)
if $F \subseteq \tilde{F}$, $\tilde{X} \subseteq X$ and $\tilde{F} \setminus F \subset X$.
A set $G$ \emph{satisfies} a Mathias condition $(F, X)$
if $F \subset G$ and $G \setminus F \subseteq X$.

\section{Degrees bounding cohesiveness}

A standard proof of Ramsey's theorem for pairs consists of
reducing an arbitrary coloring of pairs into a \emph{stable} one
using the cohesiveness principle. The understanding of the links
between cohesiveness and stability is a very active subject of research
in reverse mathematics~\cite{cholak2001strength,hirschfeldt2008strength,chongmetamathematics}.

\begin{definition}[Cohesiveness]
An infinite set $C$ is $\vec{R}$-cohesive for a sequence of sets $R_0, R_1, \dots$
if for each $i \in \omega$, $C \subseteq^{*} R_i$ or $C \subseteq^{*} \overline{R_i}$.
A set $C$ is \emph{cohesive} (resp. \emph{$r$-cohesive}) if it is $\vec{R}$-cohesive where
$\vec{R}$ is an enumeration of all c.e.\ (resp. computable) sets.
$\coh$ is the statement ``Every uniform sequence of sets $\vec{R}$
has an $\vec{R}$-cohesive set.''
\end{definition}

Jockusch \& al. proved in~\cite{jockusch1993cohesive} the existence of a low${}_2$
cohesive set. Degrees bounding $\coh$ are quite well understood and admit a simple characterization:

\begin{theorem}[Jockusch \& Stephan \cite{jockusch1993cohesive}]\label{thm:coh-bounding-char} Fix an $n \in \omega$.
\begin{itemize}
	\item[1.] For every set $C$ such that $C' \gg \emptyset'$, $C \gg_{\coh} \emptyset$.
	\item[2.] There exists a uniformly $\emptyset^{(n)}$-computable sequence of sets $\vec{R}$
	such that for every $\vec{R}$-cohesive set $C$, $(C \oplus \emptyset^{(n)})' \gg \emptyset^{(n+1)}$.
\end{itemize}
\end{theorem}

In particular, taking a set $P \gg \emptyset'$ low over $\emptyset'$ 
and a set $C$ such that $C' =_T P$ whose existence is ensured by Friedberg's jump inversion theorem, 
we obtain a low${}_2$ degree bounding $\coh$.
The canonical $\emptyset^{(n)}$-computable sequence of sets $\vec{R}$ whose existence
is claimed in clause 2 of Theorem~\ref{thm:coh-bounding-char} is
$$
R_e = \{ s : \Phi_{e,s}^{\emptyset^{(n+1)}_s}(e) \downarrow = 1 \}
$$
Every $\vec{R}$-cohesive set $C$ computes a function $f(\cdot, \cdot)$
such that $\lim_{s \in C} f(e,s)$ exists for each $e \in \omega$
and $\lim_{s \in C} f(e,s) = \Phi_e^{\emptyset^{(n+1)}}(e)$ for each
Turing index $e$ such that $\Phi_e^{\emptyset^{(n+1)}}(e) \downarrow$.
By a relativized version of Schoenfield's limit lemma, 
$(C \oplus \emptyset^{(n)})'$ computes the function $\tilde{f}(x) = \lim_{s \in C} f(x,s)$
and is therefore of PA degree relative to~$\emptyset^{(n+1)}$.

\begin{corollary}
$\coh$ admits a universal instance.
\end{corollary}
\begin{proof}
The uniformly computable sequence of sets $\vec{R}$ such that the jump of every
$\vec{R}$-cohesive set is of PA degree relative to $\emptyset'$ is a universal instance
by previous theorem.
\end{proof}

Wang proved in~\cite{wang2013cohesive} that for every set $P \gg \emptyset''$
and every uniformly $\emptyset'$-computable sequence of sets $\vec{R}$, there exists
an $\vec{R}$-cohesive set $C$ such that $C'' \leq_T C \oplus \emptyset'' \leq_T P$.
Cholak \& al.\ used in~\cite{cholak2001strength} the existence of a low subuniform degree to deduce
the existence, for every set $P \gg \emptyset'$, of an r-cohesive set $C$ such that $C' \leq_T P$.
We can apply a similar reasoning for $\emptyset'$-computable sets, using the fact that degrees bounding $\coh$
are somehow subuniform degrees for $\Delta^0_2$ approximations.

\begin{theorem}
For every set $P \gg \emptyset''$, there exists an $\vec{R}$-cohesive set $C$
such that $C'' \leq_T C \oplus \emptyset'' \leq_T P$,
where $\vec{R}$ is the (non-uniformly computable) sequence of all $\emptyset'$-computable sets.
\end{theorem}
\begin{proof}
Let $\vec{U}$ be the uniformly computable sequence of sets defined by
$$
U_{e,x} = \{s : \Phi^{\emptyset'_s}_{e,s}(x) = 1\}
$$
Fix a low${}_2$ $\vec{U}$-cohesive set $C_0$ and its $C_0$-computable bijection $f : \omega \to C_0$.
Every set $P \gg \emptyset''$, $P \gg C_0''$.
Consider the uniformly $C_0'$-computable sequence of sets
$$
V_e = \{ x : \lim_s \Phi^{\emptyset'_{f(s)}}_{e,s}(x) = 1\}
$$
The sequence $\vec{V}$ contains every $\emptyset'$-computable set.
In particular, every $\vec{V}$-cohesive set is $\vec{R}$-cohesive.
By a relativization of Wang's result, there exists an $\vec{V}$-cohesive
set $C$ such that $(C \oplus C_0)'' \leq_T C \oplus C_0'' =_T C \oplus \emptyset'' \leq_T P$.
\end{proof}

The proof of previous theorem shows that an application of $\coh$
followed by an application of $\coh[\emptyset']$
are enough to obtain a set of degree bounding $\coh[\emptyset']$.
The following question remains open:

\begin{question}
Does $\coh[\emptyset']$ admit a universal instance ?
\end{question}

\section{Degrees bounding the atomic model theorem}

The atomic model theorem is a statement of model theory admitting
a simple, purely computability theoretic characterization over $\omega$-models.
This statement happens to have a weak computational content and is therefore a consequence
of many other principles in reverse mathematics. For those reasons,
the atomic model theorem is a good candidate for factorizing proofs 
of properties which are closed upward by the consequence relation.

\begin{definition}[Atomic model theorem]
A formula $\varphi(x_1, \dots, x_n)$ of $T$ is an \emph{atom} of a theory $T$ if for each formula $\psi(x_1, \dots, x_n)$,
one of $T \vdash \varphi \imp \psi$ and $T \vdash \varphi \imp \neg \psi$ holds, but not both.
A theory $T$ is \emph{atomic} if, for every formula $\psi(x_1, \dots, x_n)$ consistent with $T$,
there exists an atom $\varphi(x_1, \dots, x_n)$ of $T$ extending it, i.e.\ one such that $T \vdash \varphi \imp \psi$.
A model $\Acal$ of $T$ is \emph{atomic} if every $n$-tuple from $\Acal$ satisfies an atom of $T$. 
$\amt$ is the statement ``Every complete atomic theory has an atomic model''.
\end{definition}

$\amt$ has been introduced as a principle by Hirschfeldt \& al.\ in \cite{hirschfeldt2009atomic}.
They proved that $\wkl$ and $\amt$ are incomparable on $\omega$-models, 
proved over $\rca$ that $\amt$ is strictly weaker than $\sads$.
The author proved in~\cite{patey2014somewhere} that $\sts(2)$ implies $\amt$ over $\rca$.
In this section we use the fact that $\amt$ is not bounded by any $\Delta^0_2$ low${}_2$
degree to deduce that none of $\amt$, $\sads$ and $\scac$ admits a universal instance.
The principle $\amt$ has been proven in~\cite{hirschfeldt2009atomic,conidis2008classifying}
to be computably equivalent to the following principle:

\begin{definition}[Escape property]
For every $\Delta^0_2$ function $f$, there exists a function~$g$ such that $f(x) \leq g(x)$ for infinitely many $x$.
\end{definition}

This equivalence does not hold over $\rca$ as, unlike $\amt$, 
the escape property implies $\ist$ over $\bst$~\cite{hirschfeldt2009atomic}.
Using this characterization, we can easily deduce the two following theorems:

\begin{theorem}[Hirschfeldt \& al. \cite{hirschfeldt2009atomic}]\label{thm:amt-bounding}
There is no low${}_2$ $\Delta^0_2$ degree bounding $\amt$.
\end{theorem}

\begin{theorem}\label{thm:amt-universal}
No principle $\Psf$ having an $\omega$-model with only low sets
and such that $\amt \leq_c \Psf$ admits a universal instance.
\end{theorem}

Theorem~\ref{thm:amt-bounding} and Theorem~\ref{thm:amt-universal} can be easily
proven using the following characterization of $\Delta^0_2$ low${}_2$ sets in terms of domination:

\begin{lemma}[Martin, \cite{martin1966classes}]\label{lem:shypzp-lowt}
A set $A \leq_T \emptyset'$ is low${}_2$ iff there exists an $f \leq_T \emptyset'$ dominating every $A$-computable function.
\end{lemma}
\begin{proof}
A set $A$ is low${}_2$ iff $\emptyset'$ is high relative to $A$.
As a set $X$ is high relative to a set $A \leq_T \emptyset'$
iff it computes a function dominating every $A$-computable function, we conclude.
\end{proof}

\begin{remark}
As explained Conidis in~\cite{conidis2008classifying},
Theorem~\ref{thm:amt-bounding} cannot be extended to every low${}_2$ sets: 
Soare~\cite{conidis2008classifying} constructed a low${}_2$ set bounding
the escape property using a forcing argument. So there exists
a low${}_2$ degree bounding $\amt$.
\end{remark}

\begin{proof}[Proof of Theorem~\ref{thm:amt-universal}]
Suppose for the sake of contradiction that $\Psf$ has a universal instance $U$
and an $\omega$-model $\Mcal$ with only low sets. As $U$ is computable, $U \in \Mcal$.
Let $X \in \Mcal$ be a (low) solution to $U$. In particular, $X$ is low${}_2$ and $\Delta^0_2$,
so by Lemma~\ref{lem:shypzp-lowt} and the computable equivalence of $\amt$ and the escape property, 
there exists a computable instance~$Y$ of~$\amt$ 
such that $X$ does not compute a solution to~$Y$. 
As $\amt \leq_c \Psf$, there exists a $Y$-computable (hence computable) instance~$Z$
of~$\Psf$ such that every solution to~$Z$ computes a solution to $Y$. Thus $X$ does not compute a solution to~$Z$,
contradicting universality of~$U$.
\end{proof}

Hirschfeldt \& al.\ proved in~\cite{hirschfeldt2007combinatorial}
the existence of an $\omega$-model of $\sads$ and $\scac$ with only low sets. 
Therefore we obtain another proof that neither $\sads$ nor $\scac$ admits a universal instance.
The result was first proven in~\cite{hirschfeldt2007combinatorial} using an ad-hoc notion of reducibility.

\begin{corollary}
None of $\amt$, $\sads$ and $\scac$ admit a universal instance.
\end{corollary}

Previous argument can not directly be applied to $\srt^2_2$, $\semo$ or $\sts(2)$
as none of those principles admit an $\omega$-model with only 
low sets~\cite{downey2001delta,kreuzer2012primitive,patey2014somewhere}.
However Lemma~\ref{thm:amt-universal} can be extended to principles such that every computable
instance has a $\Delta^0_2$ low${}_2$ solution. It is currently unknown whether
every $\Delta^0_2$ set admits a $\Delta^0_2$ low${}_2$ infinite subset in either it or its complement.
A positive answer would lead to a proof that $\srt^2_2$, $\semo$ and $\sts(2)$ have no universal instance,
and more importantly, would provide an $\omega$-model of $\srt^2_2$ not model of $\dnr[\emptyset']$
as explained in~\cite{hirschfeldt2008strength}. We shall see later that none of $\srt^2_2$,
$\semo$ and $\sts(2)$ admits a universal instance.

\section{Degrees bounding \texorpdfstring{$\sts(2)$}{STS(2)} and \texorpdfstring{$\sads$}{SADS}}

Mileti originally proved in~\cite{miletipartition} that no principle $\Psf$ having an $\omega$-model
with only low${}_2$ sets and satisfying $\srt^2_2 \leq_c \Psf$ admits a universal instance, and deduced
that none of $\srt^2_2$ and $\rt^2_2$ admit one.
In this section, we reapply his argument to much weaker statements and derive non-universality
results to a large range of principles in reverse mathematics.
Thin set theorem and ascending descending sequence are example of statements
weak enough to be a consequence of many others, and surprisingly strong enough
to diagonalize against low${}_2$ sets.

\begin{definition}[Thin set]
Let $k \in \omega$ and $f : [\omega]^k \to \omega$.
A set $A$ is \textit{thin for $f$} if $f([A]^n) \neq \omega$, that is, if the set $A$ ``avoids''
at least one color.
$\ts(k)$ is the statement ``every function $f : [\omega]^k \to \omega$ has
an infinite set thin for~$f$''.
$\sts(k)$ is the restriction of $\ts(k)$ to stable functions.
\end{definition}

Cholak \& al. studied extensively thin set principle
in \cite{cholak2001free}.
Some of the results where already stated by Friedman
without giving a proof, notably there exists an $\omega$-model of $\wkl$
which is not a model of $\ts(2)$, and the arithmetical comprehension axiom ($\aca$) does not imply $(\forall k)\ts(k)$ over~$\rca$.
Wang showed in \cite{wang2013some} that $(\forall k)\ts(k)$ does not imply $\aca$ on $\omega$-models. 
Rice \cite{ricethin} proved that $\sts(2)$ implies $\dnr$ over $\rca$.
The author proved in~\cite{patey2014somewhere} that $\rca \vdash \ts(2) \imp \rrt^2_2$.

\begin{definition}[Ascending descending sequence]
$\ads$ is the statement ``Every linear order admits an infinite ascending or descending sequence''.
$\sads$ is the restriction of $\ads$ to order types $\omega + \omega^{*}$.
\end{definition}

Tennenbaum~\cite{rosenstein1983linear} constructed a computable linear order of order type $\omega + \omega^{*}$
with no computable ascending or descending sequence. Therefore $\sads$ does not hold over $\rca$.
Hirschfeldt \& Shore~\cite{hirschfeldt2007combinatorial} studied $\ads$ within
the framework of reverse mathematics,
proving that $\ads$ imply both $\coh$ and $\bst$ over $\rca$
and that $\sads$ implies $\amt$ over $\rca$.
They constructed an $\omega$-model of $\ads$ not model of $\dnr$,
and an $\omega$-model of $\coh + \wkl$ not model of $\sads$.

The study of degrees bounding a statement 
and the existence of a universal instance are closely related.
As does Mileti in~\cite{miletipartition}, we deduce two kind of theorems
by the application of his proof technique.

\begin{theorem}\label{thm:sts-semo-sads-bounding}
There exists no low${}_2$ degree bounding any of $\sts(2)$ or $\sads$.
\end{theorem}

\begin{theorem}\label{thm:sts-semo-sads-universal}
No principle $\Psf$ having an $\omega$-model with only low${}_2$ sets
and such that any of $\sts(2)$, $\sads$ is computably reducible to $\Psf$
admits a universal instance.
\end{theorem}

The proof of the two theorems is split into three lemmas.
Lemma~\ref{lem:universal-instance} provides a general way of obtaining
bounding and universality results, assuming the ability of a principle to diagonalize
against a particular set. 
Lemma~\ref{lem:sts2-universal} and Lemma~\ref{lem:sads-universal}
state the desired diagonalization for respectively $\sts(2)$ and $\sads$.

\begin{corollary}\label{thm:rt22-sts2-universal}
None of the following principles admits a universal instance:
$\rt^2_2$, $\rwkl[\emptyset']$, $\fs(2)$, $\ts(2)$, $\cac$, $\ads$ and their
stable versions.
\end{corollary}
\begin{proof}
Each of the above mentioned principles is a consequence of $\rt^2_2$ over $\rca$
and computably implies either $\sads$ or $\sts(2)$.
See~\cite{Flood2012} for $\rwkl[\emptyset']$, \cite{cholak2001free} for~$\fs(2)$ and $\ts(2)$,
and \cite{hirschfeldt2007combinatorial} for $\cac$ and $\ads$.
By Theorem~3.1 of \cite{cholak2001strength}, there exists an $\omega$-model of $\rt^2_2$
having only low${}_2$ sets. We conclude by Theorem~\ref{thm:sts-semo-sads-universal}.
\end{proof}

In order to prove Theorem~\ref{thm:sts-semo-sads-bounding} and Theorem~\ref{thm:sts-semo-sads-universal},
we need the following theorem proven by Mileti. It simply consists of applying a relativized 
version of the low basis theorem to a $\Pi^0_1$ class of completions of the enumeration
of all partial computable sets.

\begin{theorem}[Mileti, Corollary 5.4.5 of \cite{miletipartition}]\label{thm:mileti-f-compute}
For every set $X$, there exists $f : \omega^2 \to \set{0,1}$ low over $X$ such that for every
$X$-computable set $Z$, there exists an $e \in \omega$ with $Z = \set{a \in \omega : f(e,a) = 1}$.
\end{theorem}

\begin{lemma}\label{lem:universal-instance}
Fix an $n \in \omega$ and two principles $\Psf$ and $\Qsf$ such that $\Psf \leq_c \Qsf$.
Suppose that for any $f : \omega^2 \to \set{0,1}$ satisfying $f'' \leq_T \emptyset^{(n+2)}$,
there exists a computable instance $I$ of $\Psf$ such that for each $e \in \omega$, 
if $\{a \in \omega : f(e,a) = 1\}$ is infinite then it is not a solution to~$I$. Then the following holds:
\begin{itemize}
  \item[(i)] For any degree $\mathbf{d}$ low${}_2$ over $\emptyset^{(n)}$ there is a computable instance
  $U$ of $\Psf$ such that $\mathbf{d}$ does not bound a solution to $U$.
  \item[(ii)] There is no degree low${}_2$ over $\emptyset^{(n)}$ bounding $\Psf$.
  \item[(iii)] If every computable instance $I$ of $\Qsf$ has a solution low${}_2$ over $\emptyset^{(n)}$,
  then $\Qsf$ has no universal instance.
\end{itemize}
\end{lemma}
\begin{proof}\ 
\begin{itemize}
  \item[(i)]
Consider any set $X$ of degree low${}_2$ over $\emptyset^{(n)}$.
By Theorem~\ref{thm:mileti-f-compute}, there exists a function $f : \omega^2 \to \set{0,1}$
low over $X$, hence low${}_2$ over $\emptyset^{(n)}$, such that any $X$-computable set $Z$
is of the form $\{a \in \omega : f(e,a) = 1\}$ for some $e \in \omega$. 
Take a computable instance $I$ of $\Psf$ having no solution of the form $\{a \in \omega : f(e, a) = 1\}$
for any $e \in \omega$. Then $X$ does not compute a solution to $I$.
  \item[(ii)] Immediate from (i).
  \item[(iii)] Take any computable instance $U$ of $\Qsf$. By assumption, $U$ has a solution $X$ low${}_2$ over $\emptyset^{(n)}$.
  By (i), there exists an instance $I$ of $\Psf$ such that $X$ does not compute a solution to~$I$.
	As $\Psf \leq_c \Qsf$,
	there exists an $I$-computable (hence computable) instance $J$ of $\Qsf$ such that any solution to~$J$
	computes a solution to~$I$. Then $X$ does not compute a solution to~$J$,
	hence $U$ is not a universal instance.
\end{itemize}
\end{proof}

We will prove the following lemmas which, together with Lemma~\ref{lem:universal-instance},
are sufficient to deduce Theorem~\ref{thm:sts-semo-sads-bounding}
and Theorem~\ref{thm:sts-semo-sads-universal}.

\begin{lemma}\label{lem:sts2-universal}
Fix a set $X$. Suppose $f : \omega^2 \to \set{0,1}$ satisfies $f'' \leq_T X^{''}$.
There exists an $X$-computable stable coloring $g : [\omega] \to \omega$ such that
for all $e \in \omega$, if $\{a \in \omega : f(e, a) = 1\}$ is infinite then it
is not thin for $g$.
\end{lemma}

\begin{lemma}\label{lem:sads-universal}
Fix a set $X$. Suppose $f : \omega^2 \to \set{0,1}$ satisfies $f'' \leq_T X^{''}$.
There exists a stable $X$-computable linear order $L$ such that for all $e \in \omega$,
if $\{a \in \omega : f(e, a) = 1\}$ is infinite then it is neither an ascending
nor a descending sequence in~$L$.
\end{lemma}

Before proving the two remaining lemmas, we relativize the results to colorings
over arbitrary tuples.

\begin{theorem}\label{thm:stsn-bounding}
For any $n$, there exists no degree low${}_2$ over $\emptyset^{(n)}$
bounding $\sts(n+2)$.
\end{theorem}
\begin{proof}
Apply Lemma~\ref{lem:sts2-universal} relativized to $X = \emptyset^{(n)}$
together with Lemma~\ref{lem:universal-instance}.
Simply notice that if $f : [\omega]^n \to \omega$ is a $\emptyset'$-computable
coloring, the computable coloring $g : [\omega]^{n+1} \to \omega$ obtained by an application
of Schoenfield's limit lemma is such that every infinite set thin for $g$ is thin for~$f$.
\end{proof}

\begin{theorem}\label{thm:stsn-universal}
For any $n$, no principle $\Psf$ having an $\omega$-model 
with only low${}_2$ over $\emptyset^{(n)}$ sets
and such that $\sts(n+2) \leq_c \Psf$ admits a universal instance.
\end{theorem}
\begin{proof}
Same reasoning as Theorem~\ref{thm:sts-semo-sads-universal}
using the notice in the proof of Theorem~\ref{thm:stsn-bounding}.
\end{proof}

\begin{theorem}
For any $n$, none of
$\rt^{n+2}_2$, $\rwkl[\emptyset^{(n+1)}]$, $\fs(n+2)$, $\ts(n+2)$ and their stable versions
admits a universal instance.
\end{theorem}
\begin{proof}
Fix an $n \in \omega$. Each of the above cited principles $\Psf$
satisfies $\sts(n+2) \leq_c \Psf$ and is a consequence of $\rt^{n+2}_2$ over $\omega$-models.
Cholak \& al. \cite{cholak2001strength} proved the existence of an $\omega$-model
of $\rt^{n+2}_2$ having only low${}_2$ over $\emptyset^{(n)}$ sets. 
Apply Theorem~\ref{thm:stsn-universal}.
\end{proof}

We now turn to the proofs of Lemma~\ref{lem:sts2-universal},
and Lemma~\ref{lem:sads-universal}.

\begin{proof}[Proof of Lemma~\ref{lem:sts2-universal}]
For each $e \in \omega$, let $Z_e = \set{a \in \omega : f(e,a) = 1}$.
The proof is very similar to \cite[Theorem~5.4.2.]{miletipartition}.
We build a $\emptyset'$-computable function $c : \omega \to \omega$
such that for all $e \in \omega$, if $Z_e$ is infinite then it
is not thin for~$c$. Given such a function $c$, we can then apply Schoenfield's limit lemma
to obtain a stable computable function $h : [\omega]^2 \to \omega$ such that
for each $x \in \omega$, $\lim_s h(x,s) = c(x)$.
Every set thin for $h$ is thin for $c$, and therefore for all $e \in \omega$, 
if $Z_e$ is infinite then it
is not thin for~$h$.

Suppose by Kleene's fixpoint theorem that
we are given a Turing index $d$ of the function $c$.
The construction is done by a finite injury priority argument
satisfying the following requirements for each $e, i \in \omega$:

\bigskip
$\Rcal_{e, i}$ : $Z_e$ is finite or $(\exists a)[f(e,a) = 1 \mbox{ and } \Phi^{\emptyset'}_d(a) = i]$
\bigskip

The requirements are ordered in a standard way, that is, following the pairing of the indexes.
Notice that each of these requirement is $\Sigma^f_2$, and furthermore we can effectively find an index
for each as such. Therefore, for each $e$ and $i \in \omega$, we can effectively find an integer
$m_{e,i}$ such that $R_{e,i}$ is satisfied if and only if $m_{e,i} \in f''$. 
By Schoenfield's limit Lemma relativized to $\emptyset'$ and low${}_2$-nes of $f$, 
there exists a $\emptyset'$-computable function $g : \omega^2 \to 2$ such that for all $m$,
we have $m \in f'' \biimp \lim_s g(m, s) = 1$ and $m \not \in f'' \biimp \lim_s g(m, s) = 0$.
Notice that for all $e$ and $i \in \omega$, $R_{e,i}$ is satisfied if and only if $\lim_s g(m_{e,i}, s) = 1$.

At stage $s$, assume we have defined $c(u)$ for every $u < s$.
If there exists a least strategy $\Rcal_{e,i}$ (in priority order) with $\tuple{e,i} < s$
such that $g(m_{e,i}, s) = 0$, set $c(s) = i$. Otherwise set $c(s) = 0$.
This ends the construction. We now turn to the verification.

\begin{claim}
Every requirement $\Rcal_{e,i}$ is satisfied.
\end{claim}
\begin{proof}
By induction over ordered pairs $\tuple{e,i}$
in lexicographic order. Suppose that $R_{e',i'}$ is satisfied for all $\tuple{e',i'} < \tuple{e,i}$, but
$\Rcal_{e,i}$ is not satisfied. Then there exists a threshold $t \geq \tuple{e,i}$ such that $g(m_{e',i'}, s) = 1$
for all $\tuple{e',i'} < \tuple{e,i}$ and $g(m_{e,i}, s) = 0$ whenever $s \geq t$.
By construction, $c(s) = i$ for every $s \geq t$.
As $Z_e$ is infinite, there exists an element $s \in Z_e$ such that $c(s) = i$,
so $Z_e$ is not thin for $c$ with witness $i$ and therefore $\Rcal_{e,i}$ is satisfied. Contradiction.
\end{proof}
\end{proof}

\begin{proof}[Proof of Lemma~\ref{lem:sads-universal}]
For each $e \in \omega$, let $Z_e = \set{a \in \omega : f(e,a) = 1}$.
The proof is very similar to \cite[Theorem~5.4.2.]{miletipartition}.
We build a $\Delta^0_2$ set $U$ together with a stable computable linear order $L$
such that $U$ is the $\omega$ part of $L$, that is, 
$U$ is the collection of elements $L$-below cofinitely many other elements.
We furthermore ensure that for each $e \in \omega$, if $Z_e$ is infinite, then 
it intersects both $U$ and $\overline{U}$. Therefore, if $Z_e$ is infinite,
it is neither an ascending, nor a descending sequence in $L$ as otherwise
it would be included in either $U$ or $\overline{U}$.

Assume by Kleene's fixpoint theorem that we are given the Turing index $d$ of~$U$.
The set $U$ is built by a finite injury priority construction
with the following requirements for each $e \in \omega$:

\begin{itemize}
  \item $\Rcal_{2e}$ : $Z_e$ is finite or $(\exists a)[f(e,a) = 1 \mbox{ and } \Phi^{\emptyset'}_d(a) = 1]$
  \item $\Rcal_{2e+1}$ : $Z_e$ is finite or $(\exists a)[f(e,a) = 1 \mbox{ and } \Phi^{\emptyset'}_d(a) = 0]$
\end{itemize}

Notice again that each of these requirement is $\Sigma^f_2$, and furthermore we can effectively find an index
for each as such. Therefore, for each $i \in \omega$, we can effectively find an $m_i$ such that $R_i$
is satisfied if and only if $m_i \in f''$. By two applications of Schoenfield's limit Lemma and
low${}_2$-ness of $f$, there exists a computable function $g : \omega^3 \to 2$ such that for all $m \in \omega$,
we have $m \in f'' \biimp \lim_t \lim_s g(m, s, t) = 1$ and $m \not \in f'' \biimp \lim_t \lim_s g(m, s, t) = 0$.
Notice that for all $i \in \omega$, 
$$
R_i \mbox{ is satisfied } \biimp \lim_t \lim_s g(m_i, s, t) = 1
$$

At stage $0$, $U_0 = \emptyset$ and every integer is a \emph{decision-maker} and \emph{follows} itself.
We say that $\Rcal_i$ \emph{requires attention for $u$ at stage $s$}
if $i \leq u \leq s$, $u$ is \emph{decision-maker} and $g(m_i, s, u) = 0$.
At stage $s+1$, assume we have decided $u <_L v$ or $u >_L v$ for every $u, v < s$.
Set $u <_L s$ if $u \in U_s$ and $u >_L s$ if $u \not \in U_s$.
Initially set $U_{s+1} = U_s$. For each decision-maker $u \leq s$ which has not been claimed at stage $s+1$
and for which some requirement $\Rcal_i$, $i < u$ requires attention, say that the least such $\Rcal_i$ \emph{claims} $u$
and act as follows.
\begin{itemize}
	\item[(a)] If $i = 2e$ and $u \not \in U_s$, then add $[u,s]$ to $U_{s+1}$.
	Elements of $[u+1, s]$ \emph{follow~$u$} and are no more considered as decision-makers from now on and at any further stage.
	\item[(b)] If $i = 2e+1$ and $u \in U_s$, then remove $[u,s]$ from $U_{s+1}$.
	As well, elements of $[u+1, s]$ are no more decision-makers and \emph{follow~$u$}.
\end{itemize}
The go to next decision-maker $u \leq s$. This ends the construction.
An immediate verification shows that at every stage, 
\begin{itemize}
	\item if $u$ stops being a decision-maker it never becomes again a decision-maker
	\item if $u$ follows $v$ then $v \leq u$, $v$ is a decision-maker, every $w$ between $v$ and $u$ follows $v$
	and thus $u$ will never follow any $w > v$.
\end{itemize}
So the decision-maker that $u$ follows eventually stabilizes. As well, because $g$ is limit-computable,
each decision-maker eventually stops increasing the number of followers and therefore
there are infinitely many decision-makers.

\begin{claim}
$L$ is a linear order.
\end{claim}
\begin{proof}
As $L$ is a tournament, it suffices to check there is no 3-cycle.
By symmetry, we check only the case where $u <_L s <_L v <_L u$ forms a 3-cycle
with $s$ the maximal element in $<_\omega$ order. By construction, this means that
$u \in U_s$, $v \not \in U_s$.
If $u <_\omega v$, then $u \not \in U_v$ and so there exists a decision-maker $w \leq_\omega u$ and an even number $i \leq w$
such that $\Rcal_i$ requires attention for $w$ at a stage $t \geq v$. Case (a) of the construction
applies and the interval $[w+1, t]$ is included $U$ at least until stage $s$. 
As $v \in [w+1, t]$, $v \in U_s$ contradicting our hypothesis.
Case $u >_\omega v$ is symmetric.
\end{proof}

\begin{claim}
$U$ is $\Delta^0_2$.
\end{claim}
\begin{proof}
Suppose for the sake of absurd that there exists a least element $u$
entering $U$ and leaving it infinitely many times.
Such $u$ must be a decision-maker, otherwise it would not be the least one.
Let $\Rcal_i$ be the least requirement claiming $u$ infinitely many times.
As $\lim_s g(m_i, s, u)$ exists, it will claim $u$ cofinitely many times and therefore
$u$ will be in $U$ or in $\overline{U}$ cofinitely many times. Contradiction.
\end{proof}

It immediately follows that $L$ is stable.

\begin{claim}
Every requirement $\Rcal_i$ is satisfied.
\end{claim}
\begin{proof}
By induction over $R_i$ in priority order. 
Suppose that $R_j$ is satisfied for all $j < i$, but $\Rcal_i$ is not satisfied.
Then there exists a threshold $t_0 \geq i$ such that $\lim_s g(m_j, s, t) = 1$
for all $j < i$ and $\lim_s g(m_i, s, t) = 0$ whenever $t \geq t_0$.

Then for every decision-maker $u \geq t_0$, $\Rcal_i$ will claim $u$ cofinitely many times,
and therefore $u$ will be in $U$ if $i$ is even and in $\overline{U}$ if $i$ is odd.
As every element follows the least decision-maker below itself, every $v$ above the least decision-maker
greater than $t_0$ will be in $U$ if $i$ is even and in $\overline{U}$ if $i$ is odd.
So if $Z_e$ is infinite, there will be such a $v \in Z_e$ satisfying $\Rcal_i$.
Contradiction.
\end{proof}
\end{proof}

\section{Degrees bounding the Erd\"os Moser theorem}

Another approach to the strength analysis of Ramsey's theorem for pairs
consists in seeing a coloring $f : [\omega]^2 \to 2$ as an infinite tournament
$T$ such that $T(x, y)$ holds for $x < y$ if and only if $f(x, y) = 1$.
The Erd\H{o}s Moser theorem states the existence of an infinite transitive subtournament,
that is, an infinite subset on which the tournament behaves like a linear order.
Therefore the Erd\H{o}s Moser theorem can be seen as a principle reducing instances of $\rt^2_2$
into instances of~$\ads$.

\begin{definition}[Erd\H{o}s Moser theorem]
A tournament $T$ on a domain $D \subseteq \N$ is an irreflexive binary relation on~$D$ such that for all $x,y \in D$ with $x \not= y$, exactly one of $T(x,y)$ or $T(y,x)$ holds. A tournament $T$ is \emph{transitive} if the corresponding relation~$T$ is transitive in the usual sense. A tournament $T$ is \emph{stable} if $(\forall x \in D)[(\forall^{\infty} s) T(x,s) \vee (\forall^{\infty} s) T(s, x)]$.
$\emo$ is the statement ``Every infinite tournament $T$ has an infinite transitive subtournament.''
$\semo$ is the restriction of $\emo$ to stable tournaments.
\end{definition}

Bovykin and Weiermann proved in \cite{bovykin2005strength} that $\emo + \ads$
is equivalent to $\rt^2_2$ over $\rca$, equivalence still holding between their stable versions.
Lerman \& al. \cite{lerman2013separating} proved over $\rca + \bst$ that $\emo$ implies $\opt$
and constructed an $\omega$-model of $\emo$ not model of $\srt^2_2$.
Kreuzter \& al. proved in \cite{kreuzer2012primitive} that $\semo$ implies
$\bst$ over $\rca$. Bienvenu \& al. proved in \cite{bievenu2014rwkl} that $\rca \vdash \semo \imp \rwkl$,
hence there exists an $\omega$-model of $\rrt^2_2$ not model of $\semo$.
Wang constructed in~\cite{wang2014ctfm} an $\omega$-model of $\emo + \coh$ not model of $\sts(2)$.
Finally, the author proved in~\cite{patey2014somewhere} that $\rca \vdash \emo \imp [\sts(2) \vee \coh]$.

The following notion of \emph{minimal interval}
plays a fundamental role in the analysis of $\emo$.
See~\cite{lerman2013separating} for a background analysis of $\emo$.

\begin{definition}[Minimal interval]
Let $T$ be an infinite tournament and $a, b \in T$
be such that $T(a,b)$ holds. The \emph{interval} $(a,b)$ is the
set of all $x \in T$ such that $T(a,x)$ and $T(x,b)$ hold.
Let $F \subseteq T$ be a finite transitive subtournament of $T$.
For $a, b \in F$ such that $T(a,b)$ holds, we say that $(a,b)$
is a \emph{minimal interval of $F$} if there is no $c \in F \cap (a,b)$,
i.e. no $c \in F$ such that $T(a,c)$ and $T(c,b)$ both hold.
\end{definition}

We provide in the next subsections two different proofs
of the existence of a low${}_2$ degree bounding $\emo$.
More precisely, we construct a low${}_2$ set $G$
which is, up to finite changes, transitive for every infinite computable tournament.

The author proved in~\cite{patey2014somewhere} that $[\sts(2) \vee \coh] \leq_c \emo$.
Therefore every low${}_2$ degree bounding $\emo$ bounds also $\coh$.
The proof does not seem adaptable to prove that $\coh$ is a consequence of $\emo$
even in $\omega$-models. However we can prove a weaker statement:

\begin{lemma}
For every set $X$, there exists an infinite $X$-computable tournament $T$
such that for every infinite $T$-transitive subtournament $U$,
$U \subseteq^{*} X$ or $U \subseteq^{*} \overline{X}$.
\end{lemma}
\begin{proof}
Fix a set $X$. We define a tournament $T$ as follows:
For each $a < b$, set $T(a, b)$ to hold iff $a \in X$ and $b \in X$ or $a \not \in X$ and $b \not \in X$.
Suppose for the sake of absurd that $U$ is an infinite transitive subtournament of $T$
which intersects infinitely often $X$ and $\overline{X}$.
Take any $a, c \in U \cap X$ and $b, d \in U \cap \overline{X}$ such that $a < b < c < d$.
Then $T(a, c)$, $T(c, b)$, $T(b, d)$ and $T(d, a)$ hold contradicting transitivity of $U$.
\end{proof}

Using previous lemma, the constructed set $G$ must be cohesive
and therefore provides another proof of the existence of a low${}_2$ cohesive set.
Finally, we can deduce a statement slightly weaker than Theorem~\ref{thm:stsn-bounding}
simply by the existence of a low${}_2$ degree bounding $\emo$.

\begin{lemma}
There exists a set $C$ such that
there is no low${}_2$ over $C$ degree $\textbf{d} \gg_{\sads} C$.
\end{lemma}
\begin{proof}
Fix a low${}_2$ set $C \gg_{\emo} \emptyset$
and a set $X$ low${}_2$ over $C$. By low${}_2$-ness of $C$, $X$ is low${}_2$.
Consider the stable coloring $f : [\omega]^2 \to 2$ constructed by Mileti in~\cite{miletipartition},
such that $X$ computes no infinite $f$-homogeneous set.
We can see $f$ as a stable tournament $T$ such that for each $x < y$,
$T(x,y)$ holds iff $f(x,y) = 1$. As $C \gg_{\emo} \emptyset$, there exists an infinite $C$-computable
transitive subtournament $U$ of $T$. $U$ is a stable linear order
such that every infinite ascending or descending sequence is $f$-homogeneous.
Therefore $X$ computes no infinite ascending or descending sequence in $U$.
\end{proof}

The following question remains open:

\begin{question}
Does $\emo$ admit a universal instance ?
\end{question}

\subsection{A low${}_2$ degree bounding $\emo$ using first jump control}

The following theorem uses the proof techniques introduced in~\cite{cholak2001strength} for producing
low${}_2$ sets by controlling the first jump. It is done in the same spirit
as Theorem~3.6 in~\cite{cholak2001strength}.

\begin{theorem}\label{thm:low2-degree-bounding-em-first-jump}
For every set $P \gg \emptyset'$, there exists a set $G \gg_{\emo} \emptyset$
such that $G' \leq_T P$.
\end{theorem}

Before proving Theorem~\ref{thm:low2-degree-bounding-em-first-jump}, we introduce
the notion of \emph{Erd\H{o}s Moser condition}.

\begin{definition}
An \emph{Erd\H{o}s Moser condition} (EM condition) for an infinite tournament~$T$
is a Mathias condition $(F, X)$ where
\begin{itemize}
	\item[(a)] $F \cup \{x\}$ is $T$-transitive for each $x \in X$
	\item[(b)] $X$ is included in a minimal $T$-interval of $F$.
\end{itemize}
\end{definition}

Extension is usual Mathias extension.
EM conditions have good properties for tournaments as state following lemmas.
Given a tournament $T$ and two sets $E$ and $F$,
we denote by $E \to_T F$ the formula $(\forall x \in E)(\forall y \in F) T(x,y) \mbox{ holds}$.

\begin{lemma}\label{lem:emo-cond-beats}
Fix an EM condition $(F, X)$ for a tournament $T$.
For every $x \in F$, $\{x\} \to_T X$ or $X \to_T \{x\}$.
\end{lemma}
\begin{proof}
Fix an $x \in F$. Let $(u,v)$ be the minimal $T$-interval containing $X$,
where $u, v$ may be respectively $-\infty$ and $+\infty$.
By definition of interval, $\{u\} \to_T X \to_T \{v\}$.
By definition of minimal interval, $T(x,u)$ or $T(v,x)$ holds.
Suppose the former holds.
By transitivity of $F \cup \{y\}$ for every $y \in X$, $T(x,y)$ holds, therefore $\{x\} \to_T Y$.
In the latter case, by symmetry, $Y \to_T \{x\}$.
\end{proof}

\begin{lemma}\label{lem:emo-cond-valid}
Fix an EM condition $c = (F, X)$ for a tournament $T$, 
an infinite subset $Y \subseteq X$ and a finite $T$-transitive set $F_1 \subset X$ such that
$F_1 < Y$ and $[F_1 \to_T Y \vee Y \to_T F_1]$.
Then $d = (F \cup F_1, Y)$ is a valid extension of $c$.
\end{lemma}
\begin{proof}
Properties of a Mathias condition for $d$ are immediate.
We prove property~(a). Fix an $x \in Y$.
To prove that $F \cup F_1 \cup \{x\}$ is $T$-transitive,
it suffices to check that there exists no 3-cycle in $F \cup F_1 \cup \{x\}$.
Fix three elements $u < v < w \in F \cup F_1 \cup \{x\}$.
\begin{itemize}
	\item Case 1: $\{u, v, w\} \cap F \neq \emptyset$. Then $u \in F$ as $F < F_1 < \{x\}$
	and $u < v < w$. If $v \in F$ then using the fact that $F_1 \cup \{x\} \subset X$
	and property (a) of condition $c$, $\{u, v, w\}$ is $T$-transitive.
	If $v \not \in F$, then by Lemma~\ref{lem:emo-cond-beats},
	$\{u\} \to_T X (\supseteq F \cup \{x\})$ or $X \to_T \{u\}$
	hence $\{u\} \to_T \{v,w\}$ or $\{v,w\} \to_T \{u\}$ so $\{u, v, w\}$ is $T$-transitive.
	
	\item Case 2: $\{u, v, w\} \cap F = \emptyset$. Then at least $u, v \in F_1$ because $F_1 < \{x\}$.
	If $w \in F_1$, then $\{u,v,w\}$ is $T$-transitive by $T$-transitivity of $F_1$.
	Otherwise, as $F_1 \to_T Y$ or $Y \to_T F_1$, $\{u,v\} \to_T \{w\}$
	or $\{w\} \to_T \{u,v\}$ and $\{u,v,w\}$ is $T$-transitive.
\end{itemize}
We now prove property (b).
Let $(u,v)$ be the minimal $T$-interval of $F$ in which $X$ (hence $Y$) is included by property (b) of
condition $c$. $u$ and $v$ may be respectively $-\infty$ and $+\infty$.
By assumption, either $F_1 \to_T Y$ or $Y \to_T F_1$. As $F_1$ is a finite $T$-transitive set,
it has a minimal and a maximal element, say $x$ and $y$. 
If $F_1 \to_T Y$ then $Y$ is included in the $T$-interval $(y, v)$.
Symmetrically, if $Y \to_T F_1$ then $Y$ is included in the $T$-interval $(u, x)$.
To prove minimality for the first case, assume that some $w$ is in the interval $(y, v)$.
Then $w \not \in F$ by minimality of the interval $(u,v)$ w.r.t. $F$,
and $w \not \in F_1$ by maximality of $y$.
Minimality for the second case holds by symmetry. 
\end{proof}

\begin{proof}[Proof of Theorem~\ref{thm:low2-degree-bounding-em-first-jump}]
Let $C$ be a low set such that there exists a uniformly $C$-computable enumeration
$\vec{T}$ of infinite tournaments containing every computable tournament. Note that $P \gg C'$.
Our forcing conditions are tuples $(\sigma, F, X)$ where $\sigma \in \omega^{<\omega}$ and the following holds:
\begin{itemize}
	\item[(a)] $(F,X)$ forms a Mathias condition and $X$ is a set low over $C$.
	\item[(b)] $(F \setminus [0, \sigma(\nu)], X)$ is an EM condition for $T_\nu$ for each $\nu < |\sigma|$.
\end{itemize}
A condition $(\tilde{\sigma}, \tilde{F}, \tilde{X})$ \emph{extends} a condition $(\sigma, F, X)$
if $\sigma \preceq \tilde{\sigma}$ and $(\tilde{F},\tilde{X})$ Mathias extends $(F,X)$.
A set $G$ \emph{satisfies} the condition $(\sigma, F, X)$ if
$G \setminus [0, \sigma(\nu)]$ is $T_\nu$-transitive for each $\nu < |\sigma|$ and $G$ satisfies the Mathias condition $(F,X)$.
An \emph{index} of a condition $(\sigma, F, X)$ is a code of the tuple $\tuple{\sigma, F, e}$ where $e$ is a lowness index of~$X$.

The first lemma simply states that we can ensure 
that $G$ will be infinite and eventually transitive for each tournament in $\vec{T}$.

\begin{lemma}\label{lem:infinite-em-first-jump-control}
For every condition $c = (\sigma, F, X)$ and every $i, j \in \omega$, one can
$P$-compute an extension
$(\tilde{\sigma}, \tilde{F}, \tilde{X})$ such that $|\tilde{\sigma}| \geq i$ and $|\tilde{F}| \geq j$
uniformly from $i$, $j$ and an index of~$c$. 
\end{lemma}
\begin{proof}
Let $x$ be the first element of $X$. As $X$ is low over $C$, $x$ can be found $C'$-computably
from a lowness index of $X$.
The condition $(\tilde{\sigma}, F, X)$ is a valid extension of $c$ 
where $\tilde{\sigma} = \sigma^\frown x \dots x$ so that $|\tilde{\sigma}| \geq i$.
It suffices to prove that we can $C'$-compute an extension $(\tilde{\sigma}, \tilde{F}, \tilde{X})$
with $|\tilde{F}| > |F|$ and iterate the process.
Define the computable coloring $g : X \to 2^{|\tilde{\sigma}|}$
by $g(s) = \rho$ where $\rho \in 2^{|\tilde{\sigma}|}$ such that $\rho(\nu) = 1$ iff $T_\nu(x, s)$ holds.
One can find uniformly in $P$ a $\rho \in 2^{|\tilde{\sigma}|}$ such that the following $C$-computable set is infinite:
$$
Y = \{ s \in X \setminus \{x\} : g(s) = \rho\} 
$$
By Lemma~\ref{lem:emo-cond-valid}, $((F \cup \{x\}) \setminus [0,\tilde{\sigma}(\nu)], Y)$
is a valid EM extension for $T_\nu$. As $Y$ is low over $C$,
$(\tilde{\sigma}, F \cup \{x\}, Y)$ is a valid extension for~$c$.
\end{proof}

It remains to be able to decide $e \in (G \oplus C)'$ uniformly in $e$. 
We first need to define a forcing relation.

\begin{definition}Fix a condition $c = (\sigma, F, X)$ and two integers $e$ and $x$.
\begin{itemize}
	\item[1.] $c \Vdash \Phi^{G \oplus C}_e(x) \uparrow$
	if $\Phi^{(F \cup F_1) \oplus C}_e(x) \uparrow$ for all finite subsets $F_1 \subseteq X$
	such that $F_1$ is $T_\nu$-transitive simultaneously for each $\nu < |\sigma|$.
	\item[2.] $c \Vdash \Phi^{G \oplus C}_e(x) \downarrow$
	if $\Phi^{F \oplus C}_e(x) \downarrow$.
\end{itemize}
\end{definition}

Note that the way we defined our forcing relation $c \Vdash \Psi^{G \oplus C}_e(x) \uparrow$ differs
slightly from the ``true'' forcing notion $\Vdash^{*}$ inherited by the notion of satisfaction of $G$.
The true forcing definition of this statement is the following:

$c \Vdash^{*} \Phi^{G \oplus C}_e(x) \uparrow$ if $\Phi^{(F \cup F_1) \oplus C}_e(x) \uparrow$ 
for all finite \emph{extensible} subsets $F_1 \subseteq X$
such that $F_1$ is $T_\nu$-transitive simultaneously for each $\nu < |\sigma|$, i.e.\ 
for all finite subsets $F_1 \subseteq X$ such that there exists an extension
$d = (\tilde{\sigma}, F \cup F_1, \tilde{X})$.

However $c \Vdash^{*} \Phi^{G \oplus C}_e(x) \uparrow$ is not a $\Pi^0_1$ statement
whereas $c \Vdash \Phi^{G \oplus C}_e(x) \uparrow$ is.
In particular the fact that $c \not \Vdash \Phi^{G \oplus C}_e(x) \uparrow$
does not mean that $c$ has an extension forcing its negation.
This subtlety is particularly important in Lemma~\ref{lem:force-jump-em-first-jump-control}.
The following lemma gives a sufficient constraint, namely being included in a part of a particular partition,
on finite transitive sets to ensure that they are \emph{extensible}.

\begin{lemma}\label{lem:part1-em-first-jump-control}
Let $c = (\sigma, F, X)$ be a condition and $E \subseteq X$ be a finite set. 
There exists a $2^{|\sigma|}$ partition $(E_\rho : \rho \in 2^{|\sigma|})$ of $E$ and an infinite set $Y \subseteq X$
low over $C$ such that $E < Y$ and for all $\rho \in 2^{|\sigma|}$ and $\nu < |\sigma|$,
if $\rho(\nu) = 0$ then $E_\rho \to_{T_\nu} Y$ and if $\rho(\nu) = 1$ then $Y \to_{T_\nu} E_\rho$.

Moreover this partition and a lowness index of $Y$ can be uniformly $P$-computed
from an index of $c$ and the set $E$.
\end{lemma}
\begin{proof}
Given a set $E$, define $P_E$ to be the finite set of ordered $2^{|\sigma|}$-partitions of $E$,
that is,
$$
P_E = \{ \tuple{E_\rho : \rho \in 2^{|\sigma|}} : \bigcup_{\rho \in 2^{|\sigma|}} E_\rho = E 
\mbox{ and } \rho \neq \xi \imp E_\rho \cap E_\xi = \emptyset\}
$$
Define the $C$-computable coloring $g : X \to P_E$ by $g(x) = (E^x_\rho : \rho \in 2^{|\sigma|})$ where
$E^x_\rho = \{ a \in E : (\forall \nu < |\sigma|)[T_\nu(a, x) \mbox{ holds iff } \rho(\nu) = 0]\}$.
On can find uniformly in $P$ a partition $(E_\rho : \rho \in 2^{|\sigma|})$
such that the following $C$-computable set is infinite:
$$
Y = \{ x \in X \setminus E : g(x) = (E_\rho : \rho \in 2^{|\sigma|}) \}
$$
By definition of $g$, for all $\rho \in 2^{|\sigma|}$ and $\nu < |\sigma|$,
if $\rho(\nu) = 0$ then $E_\rho \to_{T_\nu} Y$ and if $\rho(\nu) = 1$ then $Y \to_{T_\nu} E_\rho$.
\end{proof}

We are now ready to prove the key lemma of this forcing, stating
that we can $P$-decide whether or not $e \in G'$ for any $e \in \omega$.

\begin{lemma}\label{lem:force-jump-em-first-jump-control}
For every condition $(\sigma, F, X)$ and every $e \in \omega$, there exists
an extension $d = (\tilde{\sigma}, \tilde{F}, \tilde{X})$ such that one of the following holds:
\begin{itemize}
	\item[1.] $d \Vdash \Phi^{G \oplus C}_e(e) \downarrow$
	\item[2.] $d \Vdash \Phi^{G \oplus C}_e(e) \uparrow$
\end{itemize}
This extension can be $P$-computed uniformly 
from an index of $c$ and $e$.
Moreover there is a $C'$-computable procedure to decide which case holds from an index of $d$.
\end{lemma}
\begin{proof}
Let $k = |\sigma|$.
Using a $C'$-computable procedure, we can decide from an index of $c$ and $e$ whether
there exists a finite set $E \subset X$ such that for every $2^k$-partition
$(E_i : i < 2^k)$ of $E$, there exists an $i < 2^k$ and a subset $F_1 \subseteq E_i$
$T_\nu$-transitive simultaneously for each $\nu < k$ and
satisfying $\Phi_e^{(F \cup F_1) \oplus C}(e) \downarrow$.
\begin{itemize}
	\item[1.] If such a set $E$ exists, it can be $C'$-computably found. 
	By Lemma~\ref{lem:part1-em-first-jump-control}, one can $P$-computably find
	a $2^k$-partition $(E_\rho : \rho \in 2^k)$ of $E$ and a set $Y \subseteq X$ low over $C$
	such that for all $\rho \in 2^k$ and $\nu < k$,
	if $\rho(\nu) = 0$ then $E_\rho \to_{T_\nu} Y$ and if $\rho(\nu) = 1$ then $Y \to_{T_\nu} E_\rho$.
	We can $C'$-computably find a $\rho \in 2^k$ and a set $F_1 \subseteq E_\rho$
	which is $T_\nu$-transitive simultaneously for each $\nu < k$ and
	satisfying $\Phi_e^{(F \cup F_1) \oplus C}(e) \downarrow$. 
	By Lemma~\ref{lem:emo-cond-valid}, $(F \setminus [0,\sigma(\nu)]) \cup F_1, Y)$
	is a valid EM extension of $(F \setminus [0, \sigma(\nu)], X)$ for $T_\nu$ for each $\nu < k$.
	As $Y$ is low over $C$, $(\sigma, F \cup F_1, Y)$ 
	is a valid extension of~$c$ forcing $\Phi_e^{G \oplus C}(e) \downarrow$.

	\item[2.] If no such set exists, then by compactness, the $\Pi^{0,C}_1$ class
	of all $2^k$-partitions $(X_i : i < 2^k)$ of $X$ such that for every $i < 2^k$ 
	and every finite set $F_1 \subseteq X_i$ which is $T_\nu$-transitive simultaneously for each $\nu < k$,
	$\Phi_e^{(F \cup F_1) \oplus C}(e) \uparrow$ is non-empty. In other words,
	the $\Pi^{0,C}_1$ class of all $2^k$-partitions $(X_i : i < 2^k)$ of~$X$
	such that for every $i < 2^k$, $(\sigma, F, X_i) \Vdash \Phi_e^{G \oplus C}(e) \uparrow$
	is non-empty. By the relativized low basis theorem, there exists
	a $2^k$-partition $(X_i : i < 2^k)$ of $X$ low over $C$. Furthermore, a lowness index for this partition
	can be uniformly $C'$-computably found. Using~$P$, one can find
	an $i < 2^k$ such that $X_i$ is infinite.
	$(\sigma, F, X_i)$ is a valid extension of~$c$ forcing $\Phi_e^{G \oplus C}(e) \uparrow$.
\end{itemize}
\end{proof}

Using Lemma~\ref{lem:infinite-em-first-jump-control} 
and Lemma~\ref{lem:force-jump-em-first-jump-control}, one can $P$-compute an infinite
decreasing sequence of conditions $c_0 = (\epsilon, \emptyset, \omega) \geq c_1 \geq \dots$
such that for each $s > 0$
\begin{itemize}
	\item[1.] $|\sigma_s| \geq s$, $|F_s| \geq s$
	\item[2.]	$c_s \Vdash \Phi_s^{G \oplus C}(s) \downarrow$ or $c_s \Vdash \Phi_s^{G \oplus C}(s) \uparrow$
\end{itemize}
where $c_s = (\sigma_s, F_s, X_s)$.
The resulting set $G = \bigcup_s F_s$ is $T_\nu$-transitive up to finite changes for each $\nu \in \omega$
and $G' \leq_T P$.
\end{proof}

\subsection{A low${}_2$ degree bounding $\emo$ using second jump control}
We now use the second proof technique used in~\cite{cholak2001strength} for producing a low${}_2$ set.
It consists of directly controlling the second jump of the produced set.

\begin{theorem}\label{thm:low2-degree-bounding-em-second-jump}
There exists a low${}_2$ degree bounding $\emo$.
\end{theorem}
\begin{proof}
Similar to Theorem~\ref{thm:low2-degree-bounding-em-first-jump}, 
we fix a low set $C$ such that there exists a uniformly $C$-computable enumeration
$\vec{T}$ of infinite tournaments containing every computable tournament. In particular $P \gg C'$.

Our forcing conditions are the same as in Theorem~\ref{thm:low2-degree-bounding-em-first-jump}.
We can release the constraints of infinity and lowness over $C$ for $X$
in a condition $(\sigma, F, X)$. This gives the notion of \emph{precondition}.
The forcing relations extend naturally to preconditions.

\begin{definition}\label{def:em-smallness}
Fix a finite set of Turing indexes $\vec{e}$.
A condition $(\sigma, F,X)$ is \emph{$\vec{e}$-small} if there exists a number~$x$
and a sequence $(\sigma_i, F_i, X_i : i < n)$ such that for each $i < n$
\begin{itemize}
	\item[(i)] $(\sigma_i, F_i, X_i)$ is a precondition extending $c$
	\item[(ii)] $(X_i : i < n)$ is a partition of $X \cap (x,+\infty)$
	\item[(iii)] $max(X_i) < x \mbox{ or } (\sigma_i, F \cup F_i, X_i) \Vdash (\exists e \in \vec{e})(\exists y < x)\Phi_e^{G \oplus C}(y) \uparrow$
\end{itemize}
A condition is \emph{$\vec{e}$-large} if it is not $\vec{e}$-small.
\end{definition}

A condition $(\tilde{\sigma}, \tilde{F}, \tilde{X})$ is a \emph{finite extension}
of $(\sigma, F, X)$ if $\tilde{X} =^{*} X$. Finite extensions
do not play the same fundamental role as in the original forcing in~\cite{cholak2001strength}
as adding elements to the set $F$ may require to remove infinitely many elements
of the promise set~$X$ to obtain a valid extension.
We nevertheless prove the following traditional lemma.

\begin{lemma}\label{lem:em-finite-extension-second-jump}
Fix an $\vec{e}$-large condition $c = (\sigma, F, X)$.
\begin{itemize}
	\item[1.] If $\vec{e'} \subseteq \vec{e}$ then $c$ is $\vec{e'}$-large.
	\item[2.] If $d$ is a finite extension of $c$ then $d$ is $\vec{e}$-large.
\end{itemize}
\end{lemma}
\begin{proof}
Clause 1 is trivial as $\vec{e}$ appears only in a universal quantification in the definition
of $\vec{e}$-largeness. We prove clause 2.
Let $d = (\tilde{\sigma}, \tilde{F}, \tilde{X})$ be an $\vec{e}$-small finite extension of $c$.
We will prove that $c$ is $\vec{e}$-small.
Let $x \in \omega$ and $(\sigma_i, F_i, X_i : i < n)$ witness $\vec{e}$-smallness of $d$.
Let $y = max(x, X \setminus \tilde{X})$. 
For each $i < n$, set $\tilde{F}_i = (\tilde{F} \setminus F) \cup F_i$ and $\tilde{X}_i = X_i \cap (y, +\infty)$.
Then $y$ and $(\sigma_i, \tilde{F}_i, \tilde{X}_i : i < n)$ witness $\vec{e}$-smallness of $c$.
\end{proof}

\begin{lemma}\label{lem:decide-smallness-em-second-jump}
There exists a $C''$-effective procedure to decide, given an index of
a condition~$c$ and a finite set of Turing indexes $\vec{e}$,
whether $c$ is $\vec{e}$-large.
Furthermore, if $c$ is $\vec{e}$-small, there exists sets $(X_i : i < n)$ low over~$C$
witnessing this, and one may $C'$-compute a value of $n$, $x$, lowness indexes
for $(X_i : i < n)$ and the corresponding sequences $(\sigma_i, F_i, X_i : i < n)$
which witness that $c$ is $\vec{e}$-small.
\end{lemma}
\begin{proof}
Fix a condition~$c = (\sigma, F, X)$ 
The predicate ``$(\sigma, F, X)$ is $\vec{e}$-small'' can be expressed
as a $\Sigma^0_2$ statement
$$
(\exists z)(\exists Z)P(z, Z, F, X, \vec{\nu}, \vec{e})
$$
where $P$ is a $\Pi^{0,C}_1$ predicate. Here $z$ codes $n$ and $x$,
and $Z$ codes $(X_i : i < n)$. $(\exists Z)P(z, Z, F, X, \sigma, \vec{e})$
is a $\Pi^{0,C \oplus X}_1$ predicate by compactness. As $X$ is low over $C$ and $F$ and $\sigma$ are finite,
one can compute a $\Delta^{0, C}_2$ index for the same predicate $P$
with parameter $z$, an index of $c$ and $\vec{e}$,
from a lowness index for $X$, $F$ and $\sigma$.
Therefore there exists a $\Sigma^{0,C}_2$ statement with parameters
an index of $c$ and $\vec{e}$ which holds iff $c$ is $\vec{e}$-small.

If $c$ is $\vec{e}$-small, there exists sets $(X_i : i < n)$ low over $X$ (hence low over $C$) witnessing
it by the low basis theorem relativized to $C$. By the uniformity of the proof of the low basis theorem,
one can compute lowness indexes of $(X_i : i < n)$ uniformly from a lowness index of~$X$.
\end{proof}

As the extension produced in Lemma~\ref{lem:infinite-em-first-jump-control} is not a finite extension,
we need to refine it to ensure largeness preservation.

\begin{lemma}\label{lem:infinite-em-second-jump-control}
For every $\vec{e}$-large condition $c = (\sigma, F, X)$ and every $i, j \in \omega$, 
one can $P$-compute an $\vec{e}$-large extension
$(\tilde{\sigma}, \tilde{F}, \tilde{X})$ such that $\tilde{\sigma} \geq i$ and $|\tilde{F}| \geq j$
uniformly from an index of $c$, $i$, $j$ and $\vec{e}$. 
\end{lemma}
\begin{proof}
Let $x$ be the first element of $X$. As $X$ is low over $C$, $x$ can be found $C'$-computably
from a lowness index of $X$.
The condition $d = (\tilde{\sigma}, F, X)$ is a valid extension of $c$ 
where $\tilde{\sigma} = \sigma^\frown x \dots x$ so that $|\tilde{\sigma}| \geq i$.
As $d$ is a finite extension of $c$, it is $\vec{e}$-large by Lemma~\ref{lem:em-finite-extension-second-jump}.
It suffices to prove that we can $C'$-compute an $\vec{e}$-large extension $(\tilde{\sigma}, \tilde{F}, \tilde{X})$
with $|\tilde{F}| > |F|$ and iterate the process.
Define the $C$-computable coloring $g : X \to 2^{|\tilde{\sigma}|}$ as in Lemma~\ref{lem:infinite-em-first-jump-control}.
For each $\rho \in 2^{|\tilde{\sigma}|}$, define the following set:
$$
Y_\rho = \{ s \in X \setminus \{x\} : g(s) = \rho\} 
$$
There must be a $\rho \in 2^{|\tilde{\sigma}|}$ such that $Y_\rho$ is infinite 
and $(\tilde{\sigma}, F \cup \{x\}, Y_\rho)$ is $\vec{e}$-large,
otherwise the witnesses of $\vec{e}$-smallness for each $\rho \in 2^{|\tilde{\sigma}|}$ would witness $\vec{e}$-smallness of~$c$.
By Lemma~\ref{lem:decide-smallness-em-second-jump}, one can $C''$-find a $\rho \in 2^{|\tilde{\sigma}|}$ such that
$(\tilde{\sigma}, F \cup \{x\}, Y_\rho)$ is $\vec{e}$-large.
As seen in Lemma~\ref{lem:infinite-em-second-jump-control}, $(\tilde{\sigma}, F, \{x\}, Y_\rho)$ is a valid extension.
\end{proof}

The following lemma is a refinement of Lemma~\ref{lem:part1-em-first-jump-control}
controlling largeness preservation.

\begin{lemma}\label{lem:part1-em-second-jump-control}
Let $c = (\sigma, F, X)$ be an $\vec{e}$-large condition and $E \subseteq X$ be a finite set. 
There is a $2^{|\sigma|}$ partition $(E_\rho : \rho \in 2^{|\sigma|})$ of $E$ and an infinite set $Y \subseteq X$
low over $C$ such that $E < Y$ and
\begin{itemize}
	\item[1.] for all $\rho \in 2^{|\sigma|}$ and $\nu < |\sigma|$, if $\rho(\nu) = 0$ 
	then $E_\rho \to_{T_\nu} Y$ and if $\rho(\nu) = 1$ then $Y \to_{T_\nu} E_\rho$.
	\item[2.] $(\sigma, F \cup F_1, Y)$ is an $\vec{e}$-large condition extending $d$ for every $\rho \in 2^{|\sigma|}$
	and every finite set $F_1 \subseteq E_\rho$ which is $T_\nu$-transitive for each $\nu < |\sigma|$
\end{itemize}
Moreover this partition and a lowness index of $Y$ can be uniformly $C''$-computed
from an index of $c$ and the set $E$.
\end{lemma}
\begin{proof}
Given a set $E$, recall from Lemma~\ref{lem:part1-em-first-jump-control} 
that $P_E$ is the finite set or ordered $2^k$-partitions of $E$.
Define again the computable coloring $g : X \to P_E$ by $g(x) = \tuple{E^x_\rho : \rho \in 2^{|\sigma|}}$ where
$E^x_\rho = \{ a \in E : (\forall \nu < |\sigma|)[T_\nu(a, x) \mbox{ holds iff } \rho(\nu) = 0]\}$.
If for each partition $(E_\rho : \rho \in 2^{|\sigma|})$, there exists a $\rho \in 2^{|\sigma|}$
and a $F_1 \subseteq E_\rho$ which is $T_\nu$-transitive simultaneously for each $\nu < |\sigma|$
and such that $(\sigma, F \cup F_1, Y)$ is $\vec{e}$-small where
$$
Y = \{ x \in X \setminus E : g(x) = (E_\rho : \rho \in 2^{|\sigma|}) \}
$$
Then we could construct a witness of $\vec{e}$-smallness of $c$ using smallness witnesses
of $(\sigma, F \cup F_1, Y)$ for each partition $(E_\rho : \rho \in 2^{|\sigma|})$.
Therefore there must exist a partition $(E_\rho : \rho \in 2^{|\sigma|})$
such that $Y$ is infinite and $d = (\sigma, F \cup F_1, Y)$ is $\vec{e}$-large for every $\rho \in 2^{|\sigma|}$
and every $F_1 \subseteq E_\rho$ which is $T_\nu$-transitive for each $\nu < |\sigma|$.

By Lemma~\ref{lem:decide-smallness-em-second-jump}, such partition can be found $C''$-computably.
By definition of $g$, for all $\rho \in 2^{|\sigma|}$ and $\nu < k$,
if $\rho(\nu) = 0$ then $E_\rho \to_{T_\nu} Y$ and if $\rho(\nu) = 1$ then $Y \to_{T_\nu} E_\rho$.
Therefore, by Lemma~\ref{lem:emo-cond-valid}, $((F \setminus [0,\sigma(\nu)]) \cup F_1, Y)$
is a valid EM extension of $(F \setminus [0, \sigma(\nu)], X)$ for $T_\nu$ for each $\nu < |\sigma|$, 
so $d$ is a valid condition.
\end{proof}

\begin{lemma}\label{lem:em-large-to-instance}
Suppose that $c = (\sigma, F, X)$ is $\vec{e}$-large.
For every $y \in \omega$ and $e \in \vec{e}$, 
there exists an $\vec{e}$-large extension $d$
such that $d \Vdash \Phi_e^{G \oplus C}(y) \downarrow$.
Furthermore, an index for $d$ can be computed from an oracle for $C'$
from an index of $c$, $e$ and $y$.
\end{lemma}
\begin{proof}
Let $k = |\sigma|$.
As $c$ is $\vec{e}$-large, then by a compactness argument,
there exists a finite set $E \subset X$ such that for every $2^k$-partition
$(E_i : i < 2^k)$ of $E$, there exists an $i < k$
and a finite subset $F_1 \subseteq E_i$ which is $T_\nu$-transitive
simultaneously for each $\nu < k$, and $\Phi_e^{(F \cup F_1) \oplus C}(y) \downarrow$.
Moreover this set $E$ can be $C'$-computably found.
By Lemma~\ref{lem:part1-em-second-jump-control},
on can uniformly $C''$-find a partition $(E_\rho : \rho \in 2^k)$
of $E$ and a lowness index for an infinite set $Y \subseteq X$
low over~$C$ such that
\begin{itemize}
	\item[1.] for all $\rho \in 2^k$ and $\nu < k$, if $\rho(\nu) = 0$ 
	then $E_\rho \to_{T_\nu} Y$ and if $\rho(\nu) = 1$ then $Y \to_{T_\nu} E_\rho$.
	\item[2.] $(\sigma, F \cup F_1, Y)$ is an $\vec{e}$-large condition extending~$c$ for every $\rho \in 2^k$
	and every finite set finite set $F_1 \subseteq E_\rho$ which is $T_\nu$-transitive for each $\nu < k$
\end{itemize}
We can then produce by a $C'$-computable search
a $\rho \in 2^k$ and a finite set $F_1 \subseteq E_\rho$
which is $T_\nu$-transitive for each $\nu < k$ and such that
$\Phi_e^{(F \cup F_1) \oplus C}(y) \downarrow$.
By Lemma~\ref{lem:emo-cond-valid}, $((F \setminus [0,\sigma(\nu)]) \cup F_1, Y)$
is a valid EM extension of $(F \setminus [0,\sigma(\nu)], X)$ for $T_\nu$ for each $\nu < k$.
As $Y$ is low over~$C$, $(\sigma, F \cup F_1, Y)$ is a valid $\vec{e}$-large extension.
\end{proof}

\begin{lemma}\label{lem:large-small-force-em-second-jump}
Suppose that $c = (\sigma, F, X)$ is $\vec{e}$-large
and $(\vec{e} \cup \{u\})$-small. There exists a $\vec{e}$-large 
extension $d$ such that $d \Vdash \Phi_u^{G \oplus C}(y) \uparrow$ for some $y \in \omega$.
Furthermore one can find an index for $d$ by applying
a $C''$-computable function to an index of $c$, $\vec{e}$ and $u$.
\end{lemma}
\begin{proof}
By Lemma~\ref{lem:decide-smallness-em-second-jump}, we may choose the sets $(X_i : i < n)$
witnessing that $c$ is $(\vec{e} \cup \{u\})$-small
to be low over $C$. Fix the corresponding $x$ and $(\sigma_i, F_i : i < n)$.
Consider the $i$'s such that $(\sigma_i, F_i, X_i) \Vdash \Phi_u^{G \oplus C}(y) \uparrow$ for some $y < x$.
As $c$ is $\vec{e}$-large, there must be one such $i < n$ such that
$(\sigma_i, F_i, X_i)$ is an $\vec{e}$-large condition. By Lemma~\ref{lem:decide-smallness-em-second-jump}
we can find $C''$-computably such $i < n$.
$(\sigma_i, F_i, X_i)$ is the desired extension. 
\end{proof}

Using previous lemmas, we can $C''$-compute an infinite descending sequence
of conditions $c_0 = (\epsilon, \emptyset, \omega) \geq c_1 \geq \dots$
together with an infinite increasing sequence of Turing indexes $\vec{e}_0 = \emptyset \subseteq \vec{e}_1 \subseteq \dots$ 
such that for each $s > 0$
\begin{itemize}
	\item[1.] $|\sigma_s| \geq s$, $|F_s| \geq s$, $c_s$ is $\vec{e}_s$-large
	\item[2.] Either $s \in \vec{e}_s$ or $c_s \Vdash \Phi_s^{G \oplus C}(y) \uparrow$ for some $y \in \omega$
	\item[3.] $c_s \Vdash \Phi_e^{G \oplus C}(x) \downarrow$ if $s = \tuple{e,x}$ and $e \in \vec{e}_s$\end{itemize}
where $c_s = (\sigma_s, F_s, X_s)$.
The resulting set $G = \bigcup_s F_s$ is $T_\nu$-transitive up to finite changes 
simultaneously for each $\nu \in \omega$
and $G'' \leq_T C'' \leq_T \emptyset''$.
\end{proof}

\section{Degree bounding the rainbow Ramsey theorem}

The rainbow Ramsey theorem intuitively states that when a coloring over tuples
uses each color a bounded number of times then it has an infinite subset
on which each color is used at most once. This statement
has been extensively studied over the past few years~\cite{csima2009strength,conidis2013random,wang2013cohesive,patey2014somewhere}.
Remarkably, the restriction of the rainbow Ramsey theorem to coloring over pairs of integers
coincides with a well-known notion of algorithmic randomness.

\begin{definition}[Rainbow Ramsey theorem]
Let $n, k \in \omega$. A coloring function $f: [\omega]^n \to \omega$ is \emph{$k$-bounded}
  if for every $y \in \omega$, $\card{f^{-1}(y)} \leq k$. A set $R$ is a \emph{rainbow} for~$f$
  if $f \uh [R]^n$ is injective.
  $\rrt^n_k$ is the statement ``Every $k$-bounded function $f : [\omega]^n \to \omega$
  has an infinite rainbow''.
\end{definition}

A proof of the rainbow Ramsey theorem is due to Galvin who noticed that
it follows easily from Ramsey's theorem.
Hence every computable 2-bounded coloring function $f$ over $n$-tuples has an infinite $\Pi^0_n$ rainbow.
Csima and Mileti proved in \cite{csima2009strength} that every 2-random is $\rrt^2_2$-bounding
and deduced that  $\rrt^2_2$ implies neither $\sads$ nor $\wkl$ over $\omega$-models.
Conidis \& Slaman adapted in \cite{conidis2013random} the argument from Cisma and Mileti 
to obtain $\rca \vdash \ran{2} \imp \rrt^2_2$.

\begin{definition}
A function $f : \omega \to \omega$ is \emph{diagonally non-computable (DNC) relative to~$X$}
if $f(e) \neq \Phi_e^X(e)$ for each $e \in \omega$.
$\dnr[\emptyset']$ is the statement ``For every set $X$, there exists a function DNC relative to the jump of~$X$''.
\end{definition}

\begin{theorem}[J.S. Miller \cite{millercom1}]
$\rrt^2_2$ and $\dnr[\emptyset']$ are computably equivalent.
\end{theorem}

\begin{corollary}
$\rrt^2_2$ admits a universal instance.
\end{corollary}
\begin{proof}
If $\Psf$ and $\Qsf$ are two principles computably equivalent and $\Qsf$ admits a universal instance,
then so does $\Psf$. As $\dnr[\emptyset']$ admits a universal instance (any function DNC relative to $\emptyset'$),
so does $\rrt^2_2$.
\end{proof}

\begin{corollary}\label{cor:pa-degree-bounding-rrt22}
For every $X \gg \emptyset'$, there exists a $Y \gg_{\rrt^2_2} \emptyset$ such
that $Y' \leq_T X$.
\end{corollary}
\begin{proof}
Let $f : [\omega]^2 \to \omega$ be a universal instance of $\rrt^2_2$.
By Csima \& Mileti \cite{csima2009strength}, $\rrt^2_2 \leq_c \rt^2_2$,
so there exists a computable coloring $g : [\omega]^2 \to 2$
such that every infinite $g$-homogeneous set
computes an infinite $f$-rainbow, hence bounds $\rrt^2_2$.
By Cholak \& al. \cite{cholak2001strength}, for every $X \gg \emptyset'$
there exists an infinite $f$-homogeneous set $H$
such that $H' \leq_T X$. In particular $H \gg_{\rrt^2_2} \emptyset$.
\end{proof}

\begin{corollary}\label{cor:low2-degree-bounding-rrt22}
There exists a low${}_2$ degree bounding $\rrt^2_2$.
\end{corollary}
\begin{proof}
By the relativized low basis theorem, there exists a set $X \gg \emptyset'$ low over $\emptyset'$.
By Corollary~\ref{cor:pa-degree-bounding-rrt22}, there exists a set $Y \gg_{\rrt^2_2} \emptyset$
such that $Y' \leq_T X$, hence $Y'' \leq_T X' \leq_T \emptyset''$.
So $Y$ is low${}_2$.
\end{proof}

We can generalize Corollary~\ref{cor:low2-degree-bounding-rrt22} to colorings over arbitrary tuples.
For this, we need to restrict ourselves to the study of a particular class of colorings.

\begin{definition}
A coloring $f:[\omega]^{n+1} \to \omega$ is \emph{normal} if $f(\sigma, a) \neq f(\tau, b)$ 
for each $\sigma, \tau \in [\omega]^n$, whenever $a \neq b$.
\end{definition}

Wang proved in~\cite{wang2013cohesive} that for every 2-bounded coloring $f : [\omega]^n \to \omega$,
every $f$-random computes an infinite set $X$ on which $f$ is normal.
The author refined in~\cite{patey2014somewhere} this result 
by proving that every function d.n.c. relative to $f$ computes such a set.

\begin{theorem}
For each $n \geq 0$, there exists a set $X \gg_{\rrt^{n+2}_2} \emptyset$
low${}_2$ over $\emptyset^{(n)}$.
\end{theorem}
\begin{proof}
We prove by induction over~$n$ that for every set~$A$
there exists a set $X$ low${}_2$ over $A^{(n)}$ such that $X \gg_{\rrt^{n+2}_2} A$.
Case $n = 0$ is a relativization of Corollary~\ref{cor:low2-degree-bounding-rrt22}.
Suppose for each set $A$, there exists a set $X$ low${}_2$ over $A^{(n)}$
such that $X \gg_{\rrt^{n+2}_2} A$.
Fix a set $A$, an $A$-random set $R$ low over $A$
and a set $C$ low${}_2$ over $A \oplus R$ such that $C' \gg (A \oplus R)'$.
In particular $R \oplus C$ is low${}_2$ over $A$.
By induction hypothesis, there exists a set $X$ low${}_2$ over $(A \oplus R \oplus C)^{(n+1)}$
such that $X \gg_{\rrt^{n+2}_2} (A \oplus R \oplus C)'$.
In particular $X$ is low${}_2$ over $A^{(n+1)}$.
We claim that $X \gg_{\rrt^{n+3}_2} A$.

Fix an $A$-computable 2-bounded coloring $f : [\omega]^{n+3} \to \omega$.
By relativizing Lemma~4.3 in~\cite{wang2013cohesive}, every $A$-random computes
an infinite set $Y$ such that $f$ restricted to $Y$ is normal.
So $X \oplus R$ computes such a set $Y$.
For each $\sigma, \tau \in [Y]^{n+2}$, let
$$
U_{\sigma, \tau} = \{s \in Y : f(\sigma, s) = f(\tau, s)\}
$$
By Jockusch \& Frank~\cite{jockusch1993cohesive}, as $C' \gg (A \oplus R)'$,
$A \oplus R \oplus C$ computes an infinite $\vec{U}$-cohesive set $Z \subseteq Y$. 
In particular the following limit exists
$$
\tilde{f}(\sigma) = \lim_{s \in Z} min \{ \tau \leq_{lex} \sigma : f(\sigma, s) = f(\tau, s) \}
$$
$\tilde{f}$ is a 2-bounded $(A \oplus R \oplus C)'$-computable coloring of $(n+2)$-tuples,
so $X$ bounds an infinite $\tilde{f}$-rainbow $H$.
$A \oplus H$ computes an infinite $f$-rainbow, so $X$ bounds an infinite $f$-rainbow.
\end{proof}

\subsection{A stable rainbow Ramsey theorem}

A common process in the strength analysis of a principle consists
of splitting the statement into a stable and a cohesive version.
The standard notion of stability does not apply for the rainbow Ramsey theorem
as no stable coloring is $k$-bounded for some~$k \in \omega$.
Nevertheless one can define certain notions of stability for the rainbow Ramsey theorem~\cite{patey2014somewhere}.
Mileti proved in~\cite{miletipartition} that the only $\Delta^0_2$ degree bounding $\srt^2_2$ is $\mathbf{0}'$.
In fact, his priority argument can be adapted to prove the same result on a much weaker principle
coinciding with a stable version of the rainbow Ramsey theorem for pairs.

\begin{definition}
A coloring $f : [\omega]^2 \to \omega$ is 
\emph{rainbow-stable} if for every $x \in \omega$, one of the following holds:
\begin{itemize}
	\item[(a)] There exists a $y \neq x$ such that 
	$(\forall^\infty s)f(x, s) = f(y,s)$
	\item[(b)] $(\forall^{\infty} s) \card{\set{y \neq x : f(x, s) = f(y, s)}} = 0$
\end{itemize}
$\srrt^2_2$ is the statement ``every rainbow-stable 2-bounded coloring $f:[\omega]^2 \to \omega$
has a rainbow.''
\end{definition}

Introduced by the author in~\cite{patey2014somewhere},
he proved that $\srrt^2_2$ is computably reducible to $\semo$ and $\sts(2)$. This principle admits 
many computably equivalent formulations. We are particularly interested
in a characterization which can be seen as a stable notion of $\dnr[\emptyset']$.

\begin{definition}
Given a function $f : \omega \to \omega$, a function $g$ is \emph{$f$-diagonalizing}
if $(\forall x)[f(x) \neq g(x)]$.
$\sdnr[\emptyset']$ is the statement ``Every $\Delta^0_2$
function $f : \omega \to \omega$ has an $f$-diagonalizing function''.
\end{definition}

\begin{theorem}[Patey~\cite{patey2014somewhere}]
$\srrt^2_2$ and $\sdnr[\emptyset']$ are computably equivalent.
\end{theorem}

The following theorem extends Mileti's result to $\sdnr[\emptyset']$.
As $\sdnr[\emptyset']$ is computably below many stable principles,
we shall deduce a few more non-universality results.

\begin{theorem}\label{thm:incomplete-srrt22}
For every $\Delta^0_2$ incomplete set $X$,
there exists a $\Delta^0_2$ function $f : \omega \to \omega$
such that $X$ computes no $f$-diagonalizing function.
\end{theorem}

\begin{corollary}
A $\Delta^0_2$ degree $\mathbf{d}$ bounds $\srrt^2_2$ iff $\mathbf{d} = \mathbf{0}'$.
\end{corollary}
\begin{proof}
As $\srrt^2_2 \leq_c \srt^2_2$, any computable instance of $\srrt^2_2$ has a $\Delta^0_2$ solution.
So~$\mathbf{0}'$ bounds $\srrt^2_2$. If $\mathbf{d}$ is incomplete, then by Theorem~\ref{thm:incomplete-srrt22}
and by $\srrt^2_2 =_c \sdnr[\emptyset']$, there is a computable instance of $\srrt^2_2$ such that $\mathbf{d}$ bounds no solution.
\end{proof}

\begin{corollary}
No statement $\Psf$ such that $\srrt^2_2 \leq_c \Psf \leq_c \srt^2_2$ admits a universal instance.
\end{corollary}
\begin{proof}
By \cite[Corollary 4.6]{hirschfeldt2008strength} every $\Delta^0_2$ set or its complement
has an incomplete $\Delta^0_2$ infinite subset. As $\Psf \leq_c \srt^2_2 \leq_c D^2_2$,
every computable instance $U$ of $\Pcal$ has a $\Delta^0_2$ incomplete solution $X$.
By Theorem~\ref{thm:incomplete-srrt22}, there exists a 
computable coloring $f : [\omega]^2 \to \omega$ such that $X$ computes no infinite $f$-rainbow.
As $\srrt^2_2 \leq_c \Psf$, there exists a computable instance of $\Psf$ such that $X$ does not compute a solution
to it. Hence $U$ is not a universal instance of $\Psf$.
\end{proof}

\begin{corollary}
None of $\srrt^2_2$, $\semo$, $\sts(2)$ and $\sfs(2)$ admits a universal instance.
\end{corollary}

\begin{proof}[Proof of Theorem~\ref{thm:incomplete-srrt22}]
The proof is an adaptation of \cite[Theorem 5.3.7]{miletipartition}.
Suppose that $D$ is a $\Delta^0_2$ incomplete set. We will construct a $\Delta^0_2$
coloring $f : \omega \to \omega$ such that $D$ does not compute any $f$-diagonalizing function.
We want to satisfy the following requirements for each $e \in \omega$:

\bigskip
$\Rcal_e$ : If $\Phi^D_e$ is total, then there is an $a$  such that 
$\Phi^D_e(a) = f(a)$.
\bigskip

For each $e \in \omega$, define the partial function $u_e$ by letting $u_e(a)$ be the use of $\Phi_e^D$ on input
$a$ if $\Phi_e^D(a)\da$ and letting $u_e(a) \ua$ otherwise.
We can assume w.l.o.g.\ that whenever $u_e(a) \da$ then $u_e(a) \geq a$. Also define a computable partial
function $\theta$ by letting $\theta(a) = (\mu t)[a \in \emptyset'_t]$ if $a \in \emptyset'$ and $\theta(a) \ua$ otherwise.

The local strategy for satisfying a single requirement $\Rcal_e$ works as follows. If $\Rcal_e$ receives
attention at stage $s$, this strategy does the following. First it identifies a number $a \geq e$ that is \emph{not}
restrained by strategies of higher priority such that the following conditions holds:
\begin{itemize}
  \item[(i)] $\Phi_{e,s}^{D_s}(a) \downarrow$
  \item[(ii)] $u_{e,s}(a) < max(0, \theta_s(a))$
\end{itemize}
If no such number $a$ exists, the strategy does nothing. Otherwise it puts a restraint on $a$
and \emph{commits} to assigning $f_s(a) = \Phi_{e,s}^{D_s}(a)$.
For any such $a$, this commitment will remain active as long as the strategy has a restraint on this element.
Having done all this, the local strategy is declared to be satisfied and will not act again unless 
either a higher priority puts restraints on $a$, or the value of $u_{e,s}(a)$ or $\theta_s(a)$ changes.
In both cases the strategy gets \emph{injured} and has to reset, releasing all its restraints.

To finish stage $s$, the global strategy assigns values $f_s(y)$ for all $y \leq s$ as follows:
if $y$ is commited to some value assignment of $f_s(y)$ due to a local strategy, then define $f_s(y)$ to be this value.
If not, let $f_s(y) = 0$. This finishes the construction and we now turn to the verification.

For each $e, a \in \omega$, let $Z_{e, a} = \set{s \in \omega : \Rcal_e \mbox{ restrains } a \mbox{ at stage } s}$.

\begin{claim}
For each $e, a \in \omega$, 
\begin{itemize}
  \item[(a)] if $\Phi^D_e(a) \uparrow$ then $Z_{e, a}$ is finite;
  \item[(b)] if $\Phi^D_e(a) \downarrow = 1$ then $Z_{e, a}$ is either finite or cofinite.
\end{itemize}
\end{claim}
\begin{proof}
By induction on the priority order. We consider $Z_{e, a}$, assuming that for all $\Rcal_{e'}$ of higher priority,
the set $Z_{e', a}$ is either finite or cofinite. First notice that $Z_{e, a} = \emptyset$
if $a < e$ or $a \not \in \emptyset'$, so we may assume that $a \geq e$ and $a \in \emptyset'$. 
If there exists $e' < e$ such that $Z_{e', a}$
is cofinite, then $Z_{e, a}$ is finite because at most one requirement may claim $a$ at a given stage.
Suppose that $Z_{e', a}$ is finite for all $e' < e$. Fix $t_0$
such that for all $e' < e$ and $s \geq t_0$ $\Rcal_{e'}$ does not restrain $a$ at stage $s$.
and $\theta_s(a) = \theta(a)$.

Suppose that $\Phi_e^D(a) \uparrow$. Fix $t_1 \geq t_0$ such that 
$D(b) = D_s(b)$ for all $b \leq \theta(a)$ and all $s \geq t_1$.
Then for all $s \geq t_1$, if $\Phi_{e,s}^{D_s}(a) \downarrow$ 
then we must have $u_{e,s}(a) > \theta(a)$ because otherwise $\Phi^D_e(a) \downarrow$.
It follows that for all $s \geq t_1$, requirement $\Rcal_e$ does not restrain $a$ at stage $s$. Hence $Z_{e, a}$ is finite.

Suppose now that $\Phi^D_e(a) \downarrow$.
Fix $t_1 \geq t_0$ such that for all $s \geq t_1$ we have 
$\Phi^{D_s}_{e,s}(a) \downarrow$ and $D_s(c) = D(c)$ for every $c \leq u_e(a)$.
For every $s \geq t_1$, $u_{e,s}(a) = u_{e,t_1}(a)$ and $\theta_s(a) = \theta_{t_1}(a)$ for each $i \leq a$.
So properties (i) and (ii) will either hold at each stage $s \geq t_1$, or not hold at each stage $s \geq t_1$.
Therefore $Z_{e,a}$ is either finite or cofinite.
\end{proof}

\begin{claim}
Each requirement $\Rcal_e$ is satisfied. 
\end{claim}
\begin{proof}
Suppose that $\Phi_e^D$ is total for some $e \in \omega$.
We will prove that $\Phi_e^D$ is not an $f$-diagonalizing function.
Let $A = \set{a \geq e : (\forall e' < e)Z_{e', a} \mbox{ is finite}}$.
Notice that $A$ is cofinite since for each $e' < e$,
there is at most one $a$ such that $Z_{e', a}$ is cofinite. Define $h: \omega \to \omega$ as follows.

If for all but finitely many $k \in \omega$, 
we have $k \in \emptyset' \imp k \in \emptyset'_{u_e(k)}$, 
then $\emptyset' \leq_T u_e \leq_T D$, contrary to hypothesis. 
Thus we may let $a$ be the least element of $\{k \in A : k \in \emptyset' \setminus \emptyset'_{u_e(k)}\}$
greater than~$e$.
We then have
\begin{itemize}
  \item[(1)] $a \geq e$, $\Phi_e^D(a) \downarrow$, $\theta(a) > u_e(a)$
  \item[(2)] For all $e' < e$, there exists $t$ such that $\Rcal_{e'}$
  does not claim $a$ at any stage $s \geq t$.
\end{itemize}
Therefore we may fix $t \geq a$ such that for all $s \geq t$, we have $\Phi^{D_s}_{e,s}(a) \downarrow$,
$\theta_s(a) = \theta(a)$, $u_{e,s}(a) = u_e(a)$, 
and for each $e' < e$, $R_{e'}$ does not claim $a$ at stage $s$. 
Thus, for every $s \geq t$, requirement $\Rcal_e$
claims $a' \leq a$ at stage $s$. Since $Z_{e,i}$ is either finite or cofinite for each $i \leq a$,
it follows that $Z_{e,a}$ is cofinite. By the above argument, we must have
$\Phi^D_e(a) \downarrow$, and by construction, $f(a) = \Phi^D_e(a)$. Therefore $\Rcal_e$ is satisfied.
\end{proof}

\begin{claim}
The resulting function $f_s$ is $\Delta^0_2$.
\end{claim}
\begin{proof}
Consider a particular element $a$. Because of Claim 1,
if $e > a$ then $Z_{e,a} = \emptyset$. We have then two cases:
Either $Z_{e, a}$ is finite for all $e \leq a$, in which case for all but finitely many
$s$, $f_s(a) = 0$, or $Z_{e,a}$ is cofinite for some $e$. Then there is a stage $s$ at which
requirement $\Rcal_e$ has committed $f_s(a) = \Phi_e^D(a)$ for assignment
and has never been injured. Thus $f$ is $\Delta^0_2$.
\end{proof}
\end{proof}

\vspace*{1cm}

\noindent \textbf{Acknowledgements}. The author is thankful to his PhD advisor Laurent Bienvenu
for his availability during the different steps giving birth to this paper,
and his useful suggestions increasing its readability.

\bibliographystyle{plain}
\bibliography{doc}

\end{document}